\documentclass[letterpaper,11pt]{article}

\usepackage{amsmath, amsthm, amssymb, enumerate}
\usepackage{verbatim}

\theoremstyle{plain}
\newtheorem{THM}{Theorem}[section]
\newtheorem{PROP}[THM]{Proposition}
\newtheorem{LEMMA}[THM]{Lemma}

\newtheorem{CLAIM}[THM]{Claim}

\newtheorem*{QUES}{Question}

\theoremstyle{definition}
\newtheorem{DFN}[THM]{Definition}
\newtheorem{REM}[THM]{Remark}

\newcommand{\FAMILY}[4]{ \mathcal{G}(#1,#2,#3,#4) }

\newcommand{\Ex}{\mathbb{E}}

\newcommand{\cP}{\mathcal{P}}

\newcommand{\BBE}{\mathbb{E}}

\DeclareMathOperator{\Bi}{Bi}

\title{Corr\'adi and Hajnal's theorem for sparse random graphs}
\date{}
\author{
  J{\'o}zsef Balogh\thanks{Department of Mathematics, University of Illinois, 1409 W Green Street, Urbana, IL~61801, USA; and Department of Mathematics, University of California, San Diego, 9500 Gilman Drive, La Jolla, CA~92093, USA. E-mail address: jobal@math.uiuc.edu. This material is based upon work supported by NSF
CAREER Grant DMS-0745185, and OTKA Grant K76099.}
  \and
  Choongbum Lee\thanks{Department of Mathematics, UCLA, Los Angeles, CA~90095, USA. E-mail address: choongbum.lee@gmail.com. Research supported in part by Samsung Scholarship.}
  \and
  Wojciech Samotij\thanks{School of Mathematical Sciences, Tel Aviv University, Tel Aviv~69978, Israel; and Trinity College, Cambridge~CB2~1TQ, UK. E-mail address: samotij@post.tau.ac.il. Research supported in part by ERC Advanced Grant DMMCA.}
}

\begin{document}

\maketitle

\begin{abstract}
In this paper we extend a classical theorem of Corr{\'a}di and Hajnal into the setting of sparse random graphs.
We show that if $p(n) \gg (\log n / n)^{1/2}$, then asymptotically almost surely every subgraph of $G(n,p)$ with minimum degree at least $(2/3 + o(1))np$ contains a triangle packing that covers all but at most $O(p^{-2})$ vertices. Moreover, the assumption on $p$ is optimal up to the $(\log n)^{1/2}$ factor and the presence of the set of $O(p^{-2})$ uncovered vertices is indispensable.
The main ingredient in the proof, which might be of independent interest, is an embedding theorem which says
that if one imposes certain natural regularity conditions
on all three pairs in a balanced $3$-partite graph, then this graph contains a perfect triangle packing.
\end{abstract}

\section{Introduction}

\subsection{Triangle packings in subgraphs of random graphs}

Let $H$ be a fixed graph on $h$ vertices, let $G$ be a graph on $n$ vertices. An arbitrary collection of vertex-disjoint copies of $H$ in $G$ is called an \emph{$H$-packing} in $G$. A \emph{perfect $H$-packing} (an \emph{$H$-factor}) is an $H$-packing that covers all vertices of the host graph. In other words, $G$ has an {\em $H$-factor} (contains a perfect $H$-packing) if $n$ is divisible by $h$ and $G$ contains $n/h$ vertex-disjoint copies of $H$. It has been long known that for every graph $H$, if the minimum degree of $G$ is sufficiently large, then $G$ contains an $H$-factor. For example, by the Dirac's Theorem on Hamiltonian cycles~\cite{Dirac}, if $H$ is a path of length $h-1$, then $\delta(G) \geq n/2$ guarantees that $G$ has an $H$-factor. Corr{\'a}di and Hajnal~\cite{CoHa} proved that $\delta(G) \geq 2n/3$ is sufficient to guarantee a $K_3$-factor and Hajnal and Szemer{\'e}di~\cite{HaSz} showed that $\delta(G) \geq (1-1/k)n$ suffices to guarantee a $K_k$-factor for an arbitrary $k$. Moreover, all these results are easily seen to be best possible.

Finding a similar optimal condition on the minimum degree that guarantees an $H$-factor for an arbitrary graph $H$ has turned out to be significantly harder. The first result in this direction was obtained by Alon and Yuster~\cite{AlYu-almost}, who showed that $\delta(G) \geq (1-1/\chi(H))n$ implies the existence of $n/h - o(n)$ vertex-disjoint copies of $H$ in $G$. Later, the same authors~\cite{AlYu-factor} showed that $\delta(G) \geq (1 - 1/\chi(H))n + o(n)$ guarantees an $H$-factor. Finally, Koml{\'o}s, S{\'a}rk{\"o}zy, and Szemer{\'e}di~\cite{KoSaSz} showed that merely $\delta(G) \geq (1-1/\chi(H))n + c(H)$, where $c(H)$ is a (small) constant depending only on $H$, suffices. Moreover, it was observed in~\cite{AlYu-factor} that there are graphs $H$ for which the above constant $c(H)$ cannot be omitted. Recently, K{\"u}hn and Osthus~\cite{KuOs} replaced $\chi(H)$ in the above inequality by another parameter $\chi^*(H)$, which depends on the relative sizes of the color classes in the optimal colorings of $H$ and satisfies $\chi(H) - 1 < \chi^*(H) \leq \chi(H)$. Furthermore, they proved that the ratio $(1 - 1/\chi^*(H))$ in the lower bound for $\delta(G)$ is optimal for every $H$. For further information on $H$-factors in graphs with large minimum degree, we refer the reader to~\cite{KuOs09, KuOs}.

An independent direction of research concerned with $H$-factors has been determining the thresholds for the edge probability $p$ for the property that the Erd{\H{o}}s-R{\'e}nyi random graph $G(n,p)$ contains an $H$-factor. The case $H = K_2$ was solved by Erd{\H{o}}s and R{\'e}nyi~\cite{ErRe}, who proved that $\log n/n$ is the threshold for the existence of a perfect matching in $G(n,p)$. The solution for the case when $H$ is a path is a direct consequence of the result of P{\'o}sa~\cite{Posa}. Alon and Yuster~\cite{AlYu-Gnp} and, independently, Ruci{\'n}ski~\cite{Rucinski} determined the threshold for every $H$ whose fractional arboricity\footnote{The fractional arboricity of a graph $H$ is the quantity $\max\left\{\frac{|E(H')|}{|V(H')|-1}\right\}$, where the maximum is taken over all subgraphs $H'$ of $H$ with $|V(H')| > 1$.} is larger than its minimum degree. Later, partial results for the case $H = K_3$ were obtained by Krivelevich~\cite{Krivelevich} and Kim~\cite{Kim} (a related work of
Krivelevich, Sudakov, and Szab{\'o} \cite{KrSuSz} studied this case when the
host graph is a sparse pseudo-random regular graph).
Finally, Johansson, Kahn, and Vu~\cite{JoKaVu} determined the thresholds for all strictly balanced $H$ and determined them up to a sub-polynomial factor for arbitrary $H$.

Much less is known about common extensions of the results of the above two types. To make it precise, we would like to know whether it is true that for sufficiently large $p$, a.a.s.~every spanning subgraph of $G(n,p)$ with sufficiently large minimum degree has an $H$-factor. Questions like these can be naturally expressed in the framework of {\em resilience}, also called {\em fault tolerance}. Following Sudakov and Vu~\cite{SuVu}, we state the following definition.

\begin{DFN}
  Let $\cP$ be a monotone increasing graph property. The {\em local resilience} of a graph $G$ with respect to $\cP$ is the minimum number $r$ such that by deleting at most $r$ edges at each vertex of $G$, one can obtain a graph without $\cP$.
\end{DFN}

Using this terminology, one can restate, e.g., the aforementioned theorem of Corr{\'a}di and Hajnal~\cite{CoHa} by saying that the local resilience of the complete graph $K_n$ with respect to the property of having a triangle-factor is (at least) $n/3$.

Rephrasing our previous question, we would like to determine the local resilience of the random graph $G(n,p)$ with respect to the property of containing an $H$-factor for some fixed graph $H$. Sudakov and Vu~\cite{SuVu} showed that it is $(1/2+o(1))$ when $H = K_2$ and $p \gg \log n/n$ or when $H$ is a path and $p \gg (\log n)^4/n$; Lee and Sudakov~\cite{LeSu} showed that the assumption $p \gg \log n /n$ suffices also in the latter case. Recently, Huang, Lee, and Sudakov~\cite{HuLeSu} addressed this problem for an arbitrary $H$ in the case when the edge probability $p$ is a constant.

\begin{THM}
  Let $H$ be a fixed graph on $h$ vertices, let $p \in (0,1]$, and let $\gamma$ be a positive real.
  \begin{enumerate}
  \item
    If $H$ has a vertex that is not contained in any triangle, then a.a.s.~every spanning subgraph $G \subset G(n,p)$ with $\delta(G) \geq (1 - 1/\chi(H) + \gamma)np$ has a perfect $H$-packing, provided that $n$ is divisible by $h$.
  \item
    If every vertex of $H$ is contained in a triangle, then a.a.s.~every spanning subgraph $G \subset G(n,p)$ with $\delta(G) \geq (1 - 1/\chi(H) + \gamma)np$ contains an $H$-packing covering all but at most $Dp^{-2}$ vertices of $G$, where $D$ is a constant that depends only on $\chi(H)$.
  \end{enumerate}
\end{THM}

Moreover, it was shown in~\cite{HuLeSu} that in the case when each vertex of $H$ belongs to some triangle, the $Dp^{-2}$ error term cannot be removed as a.a.s.~$G(n,p)$ has a spanning subgraph with large minimum degree such that at least $\Omega(p^{-2})$ of its vertices are not contained in a triangle (and hence they are not contained in a copy of $H$). For other results on local resilience of random graphs with respect to the property of containing spanning or nearly spanning subgraphs, see~\cite{BaCsSa,BSKrSu,BoKoTa,DeKoMaSt,FrKr,KrLeSu}.

In this paper, we extend the result of Huang, Lee, and Sudakov to the sparse random graph setting in the case $H = K_3$. A rather straightforward argument using the conjecture of Kohayakawa, {\L}uczak, and R\"{o}dl~\cite[Conjecture~23]{KoLuRo}, which is known to be true for triangles, shows that if $p \gg n^{-1/2}$, then a.a.s.~every subgraph of $G(n,p)$ whose minimum degree exceeds $(2/3+o(1))np$ contains a triangle-packing that covers all but at most $\varepsilon n$ vertices, where $\varepsilon$ is an arbitrary positive constant (see Remark~\ref{rem_almost-spanning-packing}). Our main theorem proves that under the same assumptions, one can make the set of uncovered vertices significantly smaller. More precisely, we prove the following statement.

\begin{THM} \label{thm_main}
  For all positive $\gamma$, there exist constants $C$ and $D$ such that if $p \ge C(\log n /n)^{1/2}$, then a.a.s.~every subgraph $G \subset G(n,p)$ with $\delta(G) \ge (2/3 + \gamma)np$ contains a triangle packing that covers all but at most $Dp^{-2}$ vertices.
\end{THM}

Clearly, the ratio $2/3$ in the statement of Theorem~\ref{thm_main} is best possible as for every positive $\gamma$, a.a.s.~$G(n,p)$ has a subgraph $G$ with $\delta(G) \geq (2/3 - \gamma)np$ whose largest triangle packing covers no more than $(1-\gamma)n$ vertices (e.g., we may let $G$ be the intersection of $G(n,p)$ with the complete $3$-partite graph with color classes of sizes $(1+\gamma)n/3$, $n/3$, and $(1-\gamma)n/3$). Furthermore, even though it was proved in~\cite{JoKaVu} that $p \gg n^{-2/3}(\log n)^{1/3}$ guarantees that $G(n,p)$ a.a.s.~has a triangle-factor, the lower bound on $p$ in Theorem~\ref{thm_main} cannot be relaxed by more than the $(\log n)^{1/2}$ factor as if $p \ll n^{-1/2}$, then a.a.s.~one can remove all triangles from $G(n,p)$ by deleting only $o(np)$ edges incident to every vertex. Finally, the presence of the exceptional set of $Dp^{-2}$ is indispensable, see Proposition~\ref{prop_removetriangles} and~\cite[Proposition 6.3]{HuLeSu}.

\subsection{Embedding theorem for sparse regular triples}

One of the main ingredients in the proof of Theorem~\ref{thm_main} is an embedding
theorem for large triangle packings in sparse regular triples. Before we state this result
(Theorem~\ref{thm_triangleblowup} below), we recall a few basic definitions and
briefly summarize what is known about embedding large graphs into regular triples.

Let $G$ be a graph on a vertex set $V$. Given a pair of disjoint subsets $V_1, V_2 \subset V$,
let $e(V_1, V_2)$ denote the number of edges of $G$ with one endpoint in $V_1$ and the other endpoint in $V_2$,
and let the \emph{density} $d(V_1, V_2)$ of the pair $(V_1, V_2)$ be the quantity $e(V_1, V_2)/(|V_1||V_2|)$.
The pair $(V_1, V_2)$ is called \emph{$(\varepsilon,p)$-regular} if for all $V_1'
\subset V_1$ and $V_2' \subset V_2$ with $|V_1'| \ge \varepsilon
|V_1|$ and $|V_2'| \ge \varepsilon |V_2|$, we have
$\left| d(V_1, V_2) - d(V_1', V_2') \right| \le \varepsilon p$.
An $(\varepsilon,1)$-regular pair is simply called \emph{$\varepsilon$-regular}.
The concept of regularity, first developed by Szemer{\'e}di~\cite{Szemeredi}, proved to be
of extreme importance in modern combinatorics and played a central r{\^o}le in proofs of a range
of results in extremal graph theory, Ramsey theory, and others. For example,
it is well-known that every triple of sets $(V_1, V_2, V_3)$ such that $(V_i, V_j)$ is
$\varepsilon$-regular
and has sufficiently large density for all distinct $i$, $j$ contains a triangle.
An $\varepsilon$-regular pair $(V_1, V_2)$ is called \emph{$(\delta,\varepsilon)$-super-regular} if it
satisfies the additional condition that every vertex in $V_1$ has at least $\delta |V_2|$ neighbours
in $V_2$ and, vice versa, every vertex in $V_2$ has at least $\delta |V_1|$ neighbours in $V_1$.
Koml{\'o}s, S{\'a}rk{\"o}zy, and Szemer{\'e}di~\cite{KoSaSz97} proved that
super-regular triples are even more powerful than mere regular triples.
For instance, every triple $(V_1, V_2, V_3)$ such that $|V_1| = |V_2| = |V_3|$ and $(V_i, V_j)$ is
$(\delta, \varepsilon)$-super-regular and has sufficiently large density for all distinct $i$, $j$
contains not only a single triangle, but also a family of vertex-disjoint triangles that cover all
vertices of the triple.

However, if $p \ll 1$, then the power of $(\varepsilon,p)$-regular pairs turns out to be significantly weaker.
For example, {\L}uczak (see \cite{KoRo03}) observed that there are $(\varepsilon, p)$-regular triples which do not contain
even a single triangle. Still, Kohayakawa, {\L}uczak, and R\"{o}dl  \cite{KoLuRo} proved that most $(\varepsilon,p)$-regular
triples contain a triangle provided that $p$ is sufficiently large
and conjectured that an analogous result holds for arbitrary graphs (see the survey~\cite{GeSt}).

It is not much of a surprise that even less is known about embedding large graphs into sparse regular pairs. B{\"o}ttcher, Kohayakawa, and Taraz~\cite{BoKoTa} proved that if the regular pair is a subgraph of a random graph and each part has size $n$, then (asymptotically almost surely) one can embed into the pair all bipartite graphs with bounded maximum degree whose color classes both have size at most $(1-\eta)n$, where $\eta$ is a fixed positive real.
Since in an $(\varepsilon, p)$-regular pair $(V_1, V_2)$, each set $V_i$ can have as many as $c_\varepsilon|V_i|$ isolated vertices, one cannot hope to embed spanning graphs into the pair without imposing some further restrictions.
Let us now consider sparse regular triples. Observe that imposing merely a minimum degree condition as in the dense case is not sufficient since we can remove all triangles that contain a fixed vertex by deleting all edges in its neighbourhood (this will not effect regularity of the triple since the neighbourhoods of this vertex have size $o(n)$).
We suggest one possible strengthening of the notion of super-regularity, which we call \emph{strong-super-regularity}, and show that a sparse strong-super-regular triple in a subgraph of a random graph contains a collection of vertex-disjoint triangles that cover all the vertices of the triple.
The definition of a strong-super-regular triple is given in Definition~\ref{dfn_good}.


\begin{THM} \label{thm_triangleblowup}
  For all positive $\delta$ and $\xi$ there exist $\varepsilon(\delta)$ and $C(\delta,\xi)$ such that if $p(n) \geq C(\log n/n)^{1/2}$, then $G(n,p)$ a.a.s.~satisfies the following. Every $(\delta,\varepsilon,p)$-strong-super-regular triple $(V_1, V_2, V_3)$ that is a subgraph of $G(n,p)$ with $|V_1| = |V_2| = |V_3| \geq \xi n$ contains a collection of vertex-disjoint triangles that cover all the vertices.
\end{THM}

It is possible that one can derive the same conclusion from weaker assumptions than strong-super-regularity. However, we will later show that the restriction we imposed is not too strong to make our theorem useless, as Theorem~\ref{thm_triangleblowup}
will form an essential part in the proof of Theorem~\ref{thm_main}.

\subsection{Outline of the paper}

In Section~\ref{section_preliminaries}, we recall some known definitions and results and introduce a few notions that will be of great importance in all subsequent sections. Section~\ref{section_outline-proof} contains an outline of the proof of Theorem~\ref{thm_main}. In Section~\ref{section_prop-rand-graphs}, we establish some properties of the random graph $G(n,p)$ that we will frequently invoke in subsequent sections. In Sections~\ref{section_cleanup1}, \ref{section_cleanup2}, and \ref{section_packingssregular}, we prove a series of technical lemmas that culminate in the proof of Theorem~\ref{thm_main} and \ref{thm_triangleblowup}. For a brief outline of this part of the paper, we refer the reader to Section~\ref{section_outline-proof}. Finally, Section~\ref{section_concluding-remarks} contains a few concluding remarks.

\subsection{Notation}

Let $G$ be a graph with vertex set $V$ and edge set $E$. For a vertex $v \in V$, we denote its neighbourhood in $G$ by $N(v)$ and let $\deg(v)$ be its degree. The minimum degree of the graph is denoted by $\delta(G)$. For a set $X \subset V$, we let $e(X)$ be the number of edges of $G$ with both endpoints in the set $X$, and $\deg(v, X) = |N(v) \cap X|$. We say that two edges are independent if they do not share a vertex. For two subsets $X, Y \subset V$, we let $e(X,Y)$ be the number of ordered pairs $(x,y)$ such that $x \in X$, $y \in Y$ and $xy$ is an edge of $G$; note that $e(X,X) = 2e(X)$. If $X$ and $Y$ are disjoint, we refer to the quantity $e(X, Y)/(|X||Y|)$, denoted by $d(X,Y)$, as the density of the pair $(X,Y)$. With a slight abuse of notation, we will sometimes write $(X,Y)$ to denote the set of all edges $xy$ with $x \in X$ and $y \in Y$. Let $X, Y, Z \subset V$ be three pairwise disjoint sets. We say that the triple $(X, Y, Z)$ is balanced if $|X| = |Y| = |Z|$. The minimum density of the triple is the minimum of the numbers $d(X,Y)$, $d(X,Z)$, and $d(Y,Z)$. A triangle across $(X, Y, Z)$ is any triangle with one vertex in each of $X$, $Y$, and $Z$. When the implicit graph we are considering is not clear from the context, we will use subscripts to prevent ambiguity. For example, $\deg_G(v)$ is the degree of $v$ in the graph $G$.

We write $y = 1 \pm x$ to abbreviate $y \in [1-x, 1+x]$. We omit floor and ceiling signs whenever they are not crucial. Throughout the paper, $\log$ will always denote the natural logarithm. Finally, we often use subscripts such as in $c_{3.6}$ to explicitly indicate that the constant $c_{3.6}$ is defined in Claim/Lemma/Proposition/Theorem~3.6.

\section{Preliminaries}
\label{section_preliminaries}

\subsection{Sparse regularity lemma}

Let $G$ be a graph on a vertex set $V$. Recall that a pair $(V_1, V_2)$ of disjoint subsets of
$V$ is $(\varepsilon,p)$-regular if for all $V_1'
\subset V_1$ and $V_2' \subset V_2$ with $|V_1'| \ge \varepsilon
|V_1|$ and $|V_2'| \ge \varepsilon |V_2|$, $\left| d(V_1, V_2) -
d(V_1', V_2') \right| \le \varepsilon p.$ We call a triple $(V_1, V_2, V_3)$ of disjoint subsets of $V$
\emph{$(\varepsilon,p)$-regular} if $(V_i, V_j)$ forms an
$(\varepsilon,p)$-regular pair for every $\{i,j\} \subset \{1,2,3\}$. Let
$\FAMILY{K_3}{(n_1, n_2, n_3)}{(d_{12}, d_{23},
d_{31})}{(\varepsilon,p)}$ be the collection of all
$(\varepsilon,p)$-regular triples $(V_1, V_2, V_3)$
such that $|V_i| = n_i$ for all $i$ and $d(V_i, V_j) = d_{ij}p$ for
every $\{i,j\} \subset \{1,2,3\}$.

Below we establish two simple hereditary properties of regular pairs.

\begin{PROP} \label{prop_inheritregularity}
Let positive reals $\varepsilon_1$, $\varepsilon_2$, and $p$ satisfying
$\varepsilon_1 < \varepsilon_2 \leq 1/2$ be given. Let $(V_1, V_2)$ be
an $(\varepsilon_1,p)$-regular pair and for $i \in \{1,2\}$, let $V_i'
\subset V_i$ be an arbitrary subset with $|V_i'| \ge \varepsilon_2
|V_i|$. Then $(V_1', V_2')$ is an $(\varepsilon_1 / \varepsilon_2,
p)$-regular pair of density $d(V_1, V_2) \pm \varepsilon_1 p$.
\end{PROP}
\begin{proof}
By regularity of the pair $(V_1, V_2)$, for every pair of
subsets $V_i'' \subset V_i'$ such that $|V_i''| \ge
(\varepsilon_1/\varepsilon_2)|V_i'| \ge \varepsilon_1|V_i|$ for
$i \in \{1,2\}$, we have
\[
|d(V_1'', V_2'') - d(V_1', V_2')| \le |d(V_1'', V_2'') - d(V_1, V_2)|
+ |d(V_1', V_2') - d(V_1, V_2)| \le 2\varepsilon_1 p.
\]
Since $\max\{\varepsilon_1/\varepsilon_2, 2\varepsilon_1\} = \varepsilon_1/\varepsilon_2$, the pair $(V_1',
V_2')$ is $(\varepsilon_1/\varepsilon_2,p)$-regular. The density
condition immediately follows from the definition of regularity.
\end{proof}

\begin{PROP} \label{prop_inheritregularity_edgever}
Let $(V_1, V_2)$ be an $(\varepsilon, p)$-regular pair in a graph
$G$ and let $G'$ be a subgraph of $G$ obtained by removing at most
$\varepsilon^3 p|V_1||V_2|$ edges from $(V_1, V_2)$. Then $(V_1, V_2)$ is
$(3\varepsilon, p)$-regular in $G'$.
\end{PROP}
\begin{proof}
For $i \in \{1, 2\}$, let $U_i$ be a subset of $V_i$ of size at least $\varepsilon|V_i|$. Note that
\[ |d_{G}(U_1, U_2) - d_{G'}(U_1, U_2)| \le \frac{e_{G}(U_1, U_2) - e_{G'}(U_1, U_2)}{|U_1||U_2|} \le
\frac{\varepsilon^3 p|V_1||V_2|}{|U_1||U_2|} \le \varepsilon p. \]
The conclusion easily follows from the triangle inequality.
\end{proof}

An \emph{$(\varepsilon,p)$-regular partition} of an $n$-vertex graph $G$ is a partition
$(V_i)_{i=0}^k$ of its vertex set such that (i) the
exceptional class $V_0$ has size at most $\varepsilon n$, (ii)
$V_1, \ldots, V_k$ have equal sizes, and (iii) all but at most
$\varepsilon k^2$ of the pairs $(V_i, V_j)$ are
$(\varepsilon,p)$-regular. Given a collection of subsets
$(W_i)_{i=0}^{k}$ of the vertex set $V(G)$, the \emph{$(\delta,
\varepsilon, p)$-reduced graph} $R$ of the collection is the graph
on the vertex set $[k]$ such that $i,j \in [k]$ are adjacent if and
only if $W_i$ and $W_j$ form an $(\varepsilon,p)$-regular pair of
density at least $\delta p$. Note that when considering reduced
graphs, the partition $(W_i)_{i=0}^k$ is not necessarily a regular
partition and we ignore the set $W_0$. For a graph $R'$ on the
vertex set $[k]$, we say that $G$ is $(\delta,\varepsilon,p)$-regular 
over $R'$ if for every edge $\{i,j\}$ of $R'$, the pair
$(V_i,V_j)$ is $\varepsilon$-regular with density at least $\delta$.
Let $\eta$ and $b$ be reals such
that $\eta \in (0,1]$, and $b \ge 1$. We say that $G$ is \emph{$(\eta, b,
p)$-upper-uniform} if $d(V_1, V_2) \le bp$ for all disjoint sets $V_1$, $V_2$
with $|V_1|, |V_2| \ge \eta |V|$. With the above definitions at hand, we may
now state a version of Szemer{\'e}di's regularity lemma for upper-uniform graphs
(see, e.g., \cite{Kohayakawa, KoRo}). 

\begin{THM} \label{thm_regularity}
  For every positive $\varepsilon$, $b$, and $k_0$ with $b, k_0 \ge 1$, there exist constants $\eta(\varepsilon, b, k_0)$ and
$K(\varepsilon, b, k_0)$ with $K \ge k_0$ such that for every positive $p$,
every $(\eta, b, p)$-upper-uniform graph with at least $k_0$
vertices admits an $(\varepsilon,p)$-regular partition
$(V_i)_{i=0}^{k}$ such that $k_0 \le k \le K$, and each part
forms a regular pair with at least $(1 - \varepsilon)k$ other parts.
\end{THM}

The version of the regularity lemma stated above is slightly different from those given in~\cite{Kohayakawa, KoRo},
which say that the total number of irregular pairs is at most $\varepsilon k^2$. However,
by using some standard techniques, one can derive the `minimum degree' version from the results in~\cite{Kohayakawa, KoRo}.

\subsection{Typical vertices and super-regularity}

\label{subsection_typical}

We start this section by introducing the notions of typical vertices and triples.

\begin{DFN}
  \label{dfn_typical}
  Let $(V_1, V_2, V_3)$ be a triple of sets (not necessarily regular) with densities $d_{ij}p$ between $V_i$ and $V_j$.
  \begin{enumerate}[(A)]
  \item
    Fix a vertex $v \in V_1$ and for $i \in \{2, 3\}$, let $N_i = N(v) \cap V_i$. We say that $v$ is \emph{$\varepsilon$-typical} if for $i \in \{2, 3\}$,
    \begin{enumerate}[(i)]
      \setlength{\itemsep}{1pt}
      \setlength{\parskip}{0pt}
      \setlength{\parsep}{0pt}
    \item
      $|N_i| = (1 \pm \varepsilon)d_{1i}p|V_i|$ and
    \item
      there exists $N_i'\subset N_i$ satisfying $|N_i'| \ge (1 - \varepsilon)|N_i|$ such that $(N_2', N_3')$ is an $(\varepsilon,p)$-regular pair with density $(1 \pm \varepsilon)d_{23}p$.
    \end{enumerate}
  \item
    The triple $(V_1,V_2,V_3)$ is \emph{$\varepsilon$-typical} if it is $(\varepsilon,p)$-regular and for each $i$, all but at most $\varepsilon|V_i|$ vertices in $V_i$ are $\varepsilon$-typical.
  \end{enumerate}
\end{DFN}

\begin{REM}
  Since the property of being $\varepsilon$-typical depends not only on $\varepsilon$ but also on $p$, we should rather speak of $(\varepsilon, p)$-typical vertices and triples. Nevertheless, since the parameter $p$ will be always clear from the context, we will suppress it from the notation for the sake of brevity.
\end{REM}

It turns out that an overwhelming majority of all regular triples are also typical. The following lemma, which is a straightforward generalization of~\cite[Lemma 5.1]{GeSt}, makes the above statement precise. We omit its proof as it can be easily read out from the proof of~\cite[Lemma 5.1]{GeSt}.

\begin{LEMMA} \label{lemma_typicalregulartriple}
  For all positive $\beta$, $\delta$, $\varepsilon'$, and $\xi$, there exist constants $\varepsilon_0(\beta, \delta, \varepsilon')$ and $C(\delta, \varepsilon', \xi)$ such that if $\varepsilon \leq \varepsilon_0$, $d_{12}, d_{13}, d_{23} \geq \delta$, $\xi n \leq n_1, n_2, n_3 \leq n$, and $p \geq Cn^{-1/2}$, then all but at most
  \[
  \beta^{\delta\xi^2n^2p} {n_1n_2 \choose d_{12}pn_1n_2} {n_1n_3 \choose d_{13}pn_1n_3} {n_2n_3 \choose d_{23}pn_2n_3}
  \]
  graphs in $\FAMILY{K_3}{(n_1, n_2, n_3)}{(d_{12}, d_{13}, d_{23})}{(\varepsilon,p)}$ are $\varepsilon'$-typical provided that $n$ is sufficiently large.
\end{LEMMA}

The following proposition justifies why the notion of $\varepsilon$-typical triples can be useful for our purposes.

\begin{PROP} \label{prop_findtriangle}
  For every positive $\alpha$, $\delta$, and $p$, there exists an $\varepsilon(\alpha, \delta)$ such that every $\varepsilon$-typical $(\varepsilon, p)$-regular triple $(V_1, V_2, V_3)$ of minimum density at least $\delta p$ contains $(1-\alpha)\min_i|V_i|$ vertex-disjoint triangles.
\end{PROP}
\begin{proof}
  Note that without loss of generality, we may assume that $\alpha \leq 1/2$. Furthermore, let $\varepsilon = \min\{\delta/4, \alpha/20\}$ and let $\varepsilon' = 2\varepsilon/\alpha$. Let us greedily remove triangles from $(V_1, V_2, V_3)$ until we cannot do it anymore and denote the remaining triple by $(V'_1, V'_2, V'_3)$. If $|V'_i| \leq \alpha|V_i|$ for some $i$, then there is nothing left to prove, so we may assume that $|V'_i| > \alpha|V_i|$ for all $i$. Let $W_i$ be the set of all those vertices in $V'_i$ that were $\varepsilon$-typical in the original triple and note that $|W_i| \geq |V'_i| - \varepsilon|V_i| \geq (\alpha/2)|V_i|$. By Proposition~\ref{prop_inheritregularity}, the triple $(W_1, W_2, W_3)$ is $(\varepsilon', p)$-regular and the density of each pair $(W_i, W_j)$ is at least $(d_{ij} - \varepsilon)p$. Since $\varepsilon' < 1/2$, there is a vertex $v \in W_1$ with $\deg(v,W_i) \geq (d_{ij}-\varepsilon-\varepsilon')p|W_i|$ for $i \in \{2, 3\}$. For $i \in \{2, 3\}$, let $N_i$ and $N'_i$ be the sets from the definition of an $\varepsilon$-typical vertex for $v$ and let $M_i = N_i \cap W_i$ and $M'_i = N'_i \cap W_i$. Since
  \begin{align*}
    |M'_i| & \geq |M_i| - |N_i \setminus N'_i| \geq (d_{1i} - \varepsilon - \varepsilon')p|W_i| - \varepsilon(1+\varepsilon)d_{1i}p|V_i| \\
    & \geq \big[ (1 - \varepsilon/\delta - \varepsilon'/\delta)(\alpha/2) - \varepsilon(1+\varepsilon) \big] d_{1i}p|V_i| \geq \varepsilon(1+\varepsilon)d_{1i}p|V_i| \geq \varepsilon|N'_i|
  \end{align*}
  and $(N'_2, N'_3)$ was $(\varepsilon,p)$-regular with density at least $(1-\varepsilon)\delta p$ and $(1-\varepsilon)\delta p > \varepsilon p$, the pair $(M'_2, M'_3)$ has positive density. It follows that $(W_1, W_2, W_3)$ contains a triangle, but this is impossible.
\end{proof}

\begin{REM}
  \label{rem_almost-spanning-packing}
  It is quite easy to see that the combination of Theorems~\ref{thm_regularity} and~\ref{thm_CoHa}, Lemma~\ref{lemma_typicalregulartriple} (see Proposition~\ref{prop_randomgraphtypical}), and Proposition~\ref{prop_findtriangle} implies the following statement. For all positive constants $\gamma$ and $\varepsilon$, there exists a $C$ such that if $p(n) \geq Cn^{-1/2}$, then a.a.s.~every subgraph $G \subset G(n,p)$ with $\delta(G) \geq (2/3+\gamma)np$ contains a triangle packing that covers all but at most $\varepsilon n$ vertices of $G$.
\end{REM}

The following concept will serve us as a generalization of super-regularity to the sparse setting.

\begin{DFN}
  A triple $(V_1, V_2, V_3)$ is \emph{$(\delta,\varepsilon,p)$-super-regular} if each pair $(V_i, V_j)$ is $(\varepsilon, p)$-regular with density at least $\delta p$ and for every $i$, all vertices in $V_i$ are $\varepsilon$-typical.
\end{DFN}

We close this section with the following proposition, which tells us how to trim a typical regular triple in order to get a super-regular one.

\begin{PROP} \label{prop_typicalstability}
  For all positive $\varepsilon'$ and $\delta$, there exists an $\varepsilon(\varepsilon',\delta)$ such that the following holds. Let $(V_1, V_2, V_3)$ be an $(\varepsilon,p)$-regular triple, where for each $i$ and $j$, the density of $(V_i, V_j)$ is $d_{ij}p$, where $d_{ij} \geq \delta$. For each $i$, let $X_i$ be an arbitrary subset of $V_i$ with $|X_i| \leq \varepsilon|V_i|$. Then every vertex $v \in V_1 \setminus X_1$ that is $\varepsilon$-typical and satisfies $\deg(v,X_j) \leq \varepsilon p |V_j|$ for every $j \in \{2,3\}$ becomes an $\varepsilon'$-typical vertex in $(V_1 \setminus X_1, V_2 \setminus X_2, V_3 \setminus X_3)$.
\end{PROP}
\begin{proof}
  Let $v \in V_1$ be any such vertex and for $j \in \{2, 3\}$, let $N_j = N(v) \cap V_j$. Since $v$ is $\varepsilon$-typical, $(1-\varepsilon)d_{1j}p|V_j| \leq |N_j| \leq (1+\varepsilon)d_{1j}p|V_j|$. Moreover, there exist subsets $N'_j \subset N_j$ satisfying $|N'_j| \geq (1-\varepsilon)|N_j|$ such that $(N'_2,N'_3)$ is an $(\varepsilon,p)$-regular pair with density $d_vp$, where $d_v \in [(1-\varepsilon)d_{23}, (1+\varepsilon)d_{23}]$. Moreover, let $M_j = N_j \setminus X_j$ and similarly let $M'_j = N'_j \setminus X_j$.

  For each $i$, let $W_i = V_i \setminus X_i$ and recall that $(1-\varepsilon)|V_i| \leq |W_i| \leq |V_i|$. For every $i$ and $j$, let $d'_{ij}p$ be the density of the pair $(W_i, W_j)$. Since $(V_i,V_j)$ is $(\varepsilon, p)$-regular, $d'_{ij} \in [d_{ij} - \varepsilon, d_{ij} + \varepsilon]$. It follows that
  \[
  |M_j| \leq |N_j| \leq (1+\varepsilon)d_{1j}p|V_j| \leq (1+\varepsilon)(1-\varepsilon)^{-1}(d'_{1j}+\varepsilon)p|W_j|.
  \]
  And by the given condition $\deg(v, X_j) \le \varepsilon p |V_j|$,
we have
  \[
  |M_j| \geq |N_j| - \varepsilon p|V_j| \geq (1 - \varepsilon - \varepsilon/\delta)d_{1j}p|V_j| \geq (1 - \varepsilon - \varepsilon/\delta)(d'_{1j}-\varepsilon)p|W_j|.
  \]
  Moreover, since $|N_j| \geq (1-\varepsilon)\delta p|V_j|$, we have
  \begin{align*}
    |M'_j| & \geq |N'_j| - \varepsilon p |V_j| \geq (1-\varepsilon)|N_j| - \varepsilon p |V_j| \geq (1-\varepsilon-\varepsilon/((1-\varepsilon)\delta))|N_j| \\
    & \geq (1-\varepsilon-\varepsilon/((1-\varepsilon)\delta))|M_j| \geq |M_j|/2.
  \end{align*}
  By Proposition~\ref{prop_inheritregularity}, the pair $(M'_2, M'_3)$ is $(2\varepsilon,p)$-regular with density $d'_vp$ satisfying
  \[
  (1-\varepsilon)(d'_{23}-\varepsilon) - \varepsilon \leq d'_v \leq (1+\varepsilon)(d'_{23}+\varepsilon) + \varepsilon.
  \]
  Therefore, if $\varepsilon$ is sufficiently small, then $\deg(v,W_j) = |M_j| \in [(1-\varepsilon')pd'_{1j}|W_j|, (1+\varepsilon')pd'_{1j}|W_j|]$, $|M'_j| \geq (1-\varepsilon')|M_j|$, and $(M'_2, M'_3)$ is $(\varepsilon',p)$-regular with density $d'_vp$, where $d'_v \in [(1-\varepsilon')d'_{23}, (1+\varepsilon')d'_{23}]$. It follows that $v$ is $\varepsilon'$-typical in $(W_1, W_2, W_3)$.
\end{proof}

\subsection{Good edges and good vertices}

\label{subsection_good}

As we established in Section~\ref{subsection_typical} (see Remark~\ref{rem_almost-spanning-packing}), imposing certain regularity conditions on the vertices of a regular triple suffices to guarantee the existence of an almost perfect triangle packing. In order to assure that a triangle-factor can be found, we will need to impose some conditions also on the edges of the triple. With hindsight (see the discussion in Section~\ref{subsection_outline34}), we now introduce the notions of \emph{good edges} and \emph{good vertices}.

\begin{DFN}
  \label{dfn_good}
  Let $(V_1, V_2, V_3)$ be a triple of sets (not necessarily regular) with densities $d_{ij}p$ between $V_i$ and $V_j$.
  \begin{enumerate}[(A)]
  \item
    We say that an edge between $V_2$ and $V_3$ is \emph{$\varepsilon$-good} if its endpoints have at least $(1 - \varepsilon)d_{12}d_{13}p^2|V_1|$ common neighbourhoods in $V_1$.
  \item
    We say that an $\varepsilon$-typical vertex $v \in V_1$ is {\em $\varepsilon$-good} if $(N(v) \cap V_2, N(v) \cap V_3)$ contains at most $\varepsilon d_{12}d_{13}d_{23}p^3|V_2||V_3|$ edges that are not $\varepsilon$-good.
  \item
    We say that the triple $(V_1, V_2, V_3)$ is \emph{$(\delta,\varepsilon,p)$-strong-super-regular} if it is $(\delta,\varepsilon,p)$-super-regular and for every $i$, all vertices in $V_i$ are $\varepsilon$-good.
  \end{enumerate}
\end{DFN}

Next, we show that super-regular triples are not very far from being strong-super-regular. More precisely, we prove that requiring a triple to be merely typical (recall Definition~\ref{dfn_typical}) and all pairs in this triple to have non-zero densities forces most of its edges and vertices to be good.

\begin{PROP}
  \label{prop_goodedges}
Let $\varepsilon$, $\varepsilon'$, and $\delta$ be positive constants satisfying $3\varepsilon + \varepsilon/\delta \leq \varepsilon'$ and let $(V_1, V_2, V_3)$ be an $\varepsilon$-typical $(\varepsilon,p)$-regular triple, where the density $d_{ij}p$ of each pair $(V_i, V_j)$ is at least $\delta p$. Then there are at most $4\varepsilon d_{23}p|V_2||V_3|$ edges between $V_2$ and $V_3$ which are not $\varepsilon'$-good.
\end{PROP}
\begin{proof}
For an $\varepsilon$-typical vertex $v \in V_2$, let $N_1 = N(v) \cap V_1$ and $N_3 = N(v) \cap V_3$. Recall from Definition~\ref{dfn_typical} that $|N_i| \geq (1-\varepsilon)d_{2i}p|V_i|$ and that there exist $N_i' \subset N_i$ with $|N_i'| \ge (1-\varepsilon)|N_i|$ for $i \in \{1,3\}$ such that $(N_1', N_3')$ is $(\varepsilon,p)$-regular and has density at least $(1-\varepsilon)d_{13}p$. It follows that at least $(1-\varepsilon)|N_3'|$ vertices $w \in N_3'$ have at least $((1-\varepsilon)d_{13} - \varepsilon)p|N_1'|$ common neighbours with $v$ in $N_1'$. Since
\[
((1-\varepsilon)d_{13} - \varepsilon)p|N_1'| \ge (1-\varepsilon-\varepsilon/\delta)(1-\varepsilon)d_{13}p|N_1| \ge (1-3\varepsilon-\varepsilon/\delta)d_{12}d_{13}p^2|V_1|,
\]
and $3\varepsilon + \varepsilon/\delta \le \varepsilon'$, each such edge $\{v,w\}$ is $\varepsilon'$-good. Since there are at least $(1-\varepsilon)|V_2|$ typical vertices in $V_2$ and
\[
(1-\varepsilon)|N_3'| \geq (1-\varepsilon)^3d_{23}p|V_3| \geq (1-3\varepsilon)d_{23}p|V_3|,
\]
the total number of $\varepsilon'$-good edges between $V_2$ and $V_3$ is at least $(1-4\varepsilon)d_{23}p|V_2||V_3|$. Finally, since the number of edges between $V_2$ and $V_3$ is exactly $d_{23}p|V_2||V_3|$, the total number of non-$\varepsilon'$-good edges is at most $4\varepsilon d_{23}p|V_2||V_3|$.
\end{proof}

\begin{PROP}
  \label{prop_manygoodvertices}
  For every $\varepsilon'$ and $\delta$, there exists a positive $\varepsilon(\varepsilon', \delta)$ such that the following holds. Let $(V_1, V_2, V_3)$ be an $\varepsilon$-typical $(\varepsilon, p)$-regular triple with minimum density at least $\delta p$. Moreover, assume that the endpoints of no edge in $(V_2, V_3)$ have more than $4p^2|V_1|$ common neighbours in $V_1$. Then $V_1$ contains at most $\varepsilon'|V_1|$ vertices that are not $\varepsilon'$-good.
\end{PROP}
\begin{proof}
  Let $\varepsilon = \min\{\varepsilon'/4, \varepsilon'\delta/4, (\varepsilon'\delta)^2/32\}$. By Proposition~\ref{prop_goodedges}, at most $4\varepsilon d_{23}p|V_2||V_3|$ edges in $(V_2,V_3)$ are not $\varepsilon'$-good. Let $\alpha|V_1|$ be the number of $\varepsilon$-typical vertices in $V_1$ that are not $\varepsilon'$-good. By definition, the neighbourhood of every such vertex contains at least $\varepsilon'd_{12}d_{13}d_{23}p^3|V_2||V_3|$ edges that are not $\varepsilon'$-good. Therefore, our assumption on the maximum number of common neighbours of the endpoints of edges in $(V_2, V_3)$ implies that
  \[
  \alpha|V_1| \cdot \varepsilon'd_{12}d_{13}d_{23}p^3|V_2||V_3| \leq 4p^2|V_1| \cdot 4\varepsilon d_{23}p|V_2||V_3|
  \]
and hence $\alpha \leq 16\varepsilon/(\varepsilon'd_{12}d_{13}) \leq
\varepsilon'/2$. Finally, since at most $\varepsilon|V_1|$ vertices
in $V_1$ are not $\varepsilon$-typical and $\varepsilon \leq
\varepsilon'/2$, the number of vertices in $V_1$ that are not
$\varepsilon'$-good is at most $\varepsilon'|V_1|$.
\end{proof}

We end this section by showing that the neighbourhood of every typical (good) vertex contains a subgraph with bounded maximum degree and many (good) edges.

\begin{PROP} \label{prop_goodtriangles}
  Let $\varepsilon$, $\delta$, and $p$ be positive constants with
  $\varepsilon < 1/2$. Let $(V_1, V_2, V_3)$ be a triple of sets such
  that for all $i$ and $j$, the density of $(V_i, V_j)$ is $d_{ij}p$, where
  $d_{ij} \geq \delta$. Then for every $\varepsilon$-typical vertex $v \in V_1$,
  there exist sets $N''_j \subset N(v) \cap V_j$ for $j \in \{2,3\}$
  such that
  \begin{enumerate}[(i)]
    \setlength{\itemsep}{1pt}
    \setlength{\parskip}{0pt}
    \setlength{\parsep}{0pt}
  \item \label{item_goodtriangles_i}
    there are at least $(1-6\varepsilon-2\varepsilon/\delta)d_{12}d_{13}d_{23}p^3|V_2||V_3|$ edges in $(N''_2,N''_3)$ and if $v$ is $\varepsilon$-good, then there are at least that many $\varepsilon$-good edges in $(N''_2, N''_3)$, and
  \item \label{item_goodtriangles_ii}
    for all $j$ and $k$ with $\{j,k\} = \{2,3\}$, no vertex in $N''_j$ has more than $(1+2\varepsilon/\delta)(1+\varepsilon^2)d_{1k}d_{23}p^2|V_k|$ neighbours in $N''_k$.
  \end{enumerate}
\end{PROP}
\begin{proof}
Fix an $\varepsilon$-typical vertex $v \in V_1$. For each $j \in
\{2,3\}$, let $N_j = N(v) \cap V_j$. Since $v$ is
$\varepsilon$-typical, there are $N'_j \subset N_j$ with $|N'_j|
\geq (1-\varepsilon)|N_j| \geq (1-\varepsilon)^2pd_{1j}|V_j|$ such
that $(N'_2,N'_3)$ is $(\varepsilon,p)$-regular with density
$d'_{23}p$, where $(1-\varepsilon)d_{23} \leq d'_{23} \leq
(1+\varepsilon)d_{23}$. Let $N''_j$ be the set of vertices in $N'_j$
that have at most $(d'_{23}+\varepsilon)p|N'_k|$ neighbours in
$N'_k$ and note that $|N''_j| \geq (1-\varepsilon)|N'_j|$ by
$(\varepsilon, p)$-regularity of $(N'_2, N'_3)$.
Since $1/d'_{23} \leq 1/((1-\varepsilon)d_{23}) \leq 2/\delta$, we have
  \[
  (d'_{23}+\varepsilon)p|N'_k| \leq (1+2\varepsilon/\delta)d'_{23}p|N_k| \leq (1+2\varepsilon/\delta)(1+\varepsilon)^2 d_{1k}d_{23}p^2|V_k|,
  \]
  and then (\ref{item_goodtriangles_ii}) follows. Note that $|N''_j| \geq (1-\varepsilon)|N'_j| \geq (1-\varepsilon)^2d_{1j}p|V_j|$ and
  \[
  e(N''_2,N''_3) \geq (d'_{23}-\varepsilon)p |N''_2||N''_3| \geq (1-2\varepsilon/\delta)d'_{23}p|N''_2||N''_3| \geq (1-2\varepsilon/\delta)(1-\varepsilon)^5 d_{12}d_{13}d_{23} p^3 |V_2||V_3|.
  \]
  Moreover, if $v$ is $\varepsilon$-good, then at most $\varepsilon d_{12}d_{13}d_{23}p^3|V_2||V_3|$ edges in $(N''_2,N''_3)$ are not $\varepsilon$-good. Now (\ref{item_goodtriangles_i}) follows.
\end{proof}

\subsection{Graph theory}

The following proposition, which we will be using several times in the proof of our main result, is a simple corollary from Hall's marriage theorem~\cite{Hall} and gives a sufficient condition for a bipartite graph to have a perfect matching.

\begin{PROP} \label{prop_Hall}
  Let $H$ be a bipartite graph on the vertex set $A \cup B$ with $|A| = |B|$. Suppose that there is an integer $L$ such that
  \begin{enumerate}[(i)]
  \setlength{\itemsep}{1pt}
  \setlength{\parskip}{0pt}
  \setlength{\parsep}{0pt}
  \item
    $|N(S)| \geq |S|$ for each $S \subset A$ with $|A \setminus S| \geq L$ and
  \item
    $|N(T)| \geq |T|$ for each $T \subset B$ with $|T| \leq L$.
  \end{enumerate}
  Then $H$ has a perfect matching.
\end{PROP}

Recall that the following theorem was proved by Corr{\'a}di and Hajnal~\cite{CoHa}.

\begin{THM} \label{thm_CoHa}
  Every graph on $n$ vertices with minimum degree at least $2n/3$ contains a perfect $K_3$-packing provided that $n$ is divisible by $3$.
\end{THM}

\subsection{Bounding large deviations}

Throughout the proof, we will extensively use the following standard estimate on the tail probabilities of binomial random variables, see~\cite[Appendix A]{AlSp}. We denote by $\Bi(n,p)$ the binomial random variable with parameters $n$ and $p$, i.e., the number of successes in a sequence of $n$ independent Bernoulli trials with success probability $p$.

\begin{THM}[Chernoff's inequality]
  Let $p \in (0,1)$ and let $n$ be a positive integer. Then for every positive $a$ with $a \leq 2np/3$,
  \[
  P\big(|\Bi(n,p) - np| > a\big) \leq \exp(-a^2/(6np)).
  \]
\end{THM}

\section{Outline of the proof of Theorem~\ref{thm_main}}
\label{section_outline-proof}

Let $G$ be a subgraph of $G(n,p)$ with minimum degree at least $(2/3 + o(1))np$. Throughout this section, we will tacitly condition on a few events that hold in $G(n,p)$ asymptotically almost surely. The proof of Theorem~\ref{thm_main} breaks down into the following four simple steps.

\begin{enumerate}
\item
  Apply the sparse regularity lemma (Theorem~\ref{thm_regularity}) and Theorem~\ref{thm_CoHa} to partition the vertex set of $G$ into regular triples with positive density and a small exceptional set of vertices.
\item
  Remove from $G$ a collection of vertex-disjoint triangles so that all but at most $O(p^{-2})$ remaining vertices lie in balanced super-regular triples.
\item
  Decompose each of those super-regular triples into a triangle packing, a balanced strong-super-regular triple, and a set of $O(p^{-1})$ leftover vertices.
\item
  Find a triangle-factor in each strong-super-regular triple.
\end{enumerate}

Since step 1 is a straightforward application of the regularity lemma (Theorem~\ref{thm_regularity}) and Theorem~\ref{thm_CoHa},
we will only describe the basic ideas of steps 2, 3, and 4 in this section. The details
of these steps will be given in Sections \ref{section_cleanup1}, \ref{section_cleanup2},
and \ref{section_packingssregular}, respectively.

\subsection{Step 2}

\label{subsection_outline2}

In order to construct super-regular triples from the regular triples we obtained in step 1, we first move all non-typical vertices to the exceptional set $V_0$. Since we have no control over $V_0$ and $|V_0|$ can be linear in $n$, we need to cover most of it with vertex-disjoint triangles. At the same time, we do not want to use too many vertices from any of the regular triples in order not to destroy their structure, i.e., to keep them close to being super-regular. This will be achieved by an application of Lemma~\ref{lemma_leftovertriangles}, which allows us to find such triangles. After we absorb the exceptional vertices into a triangle packing, some triples in the remaining graph might become imbalanced. Since in order for any triple to have a triangle-factor (or at least an almost perfect triangle packing), the sizes of all three of its parts must be equal, we have to balance the sizes of the remaining triples. We will do that by adding to our triangle packing some triangles whose vertices lie in two different triples, see Lemma~\ref{lemma_balancingpartition}.  Finally, since at the beginning we removed all non-typical vertices from each triple and later we did not alter it too much, we can make every triple super-regular by deleting at most $O(p^{-1})$ of its vertices (see Proposition~\ref{prop_typicalstability}).

\subsection{Steps 3 and 4}

\label{subsection_outline34}

Our general strategy for finding a triangle-factor in a super-regular triple $(V_1, V_2, V_3)$ can be summarized as follows.

\begin{enumerate}[(i)]
\item
  For each $\{i, j\} \subset \{1, 2, 3\}$, randomly select a small set $M_{ij}$ of independent edges in $(V_i, V_j)$.
\item
  Find an almost perfect triangle packing that does not hit any endpoints of the edges in any $M_{ij}$.
\item
  Match the remaining vertices with the edges in the sets $M_{ij}$ in order to extend the triangle packing to a triangle-factor.
\end{enumerate}

Assume that the first two steps have been performed. Then, in order to verify Hall's condition (see Proposition~\ref{prop_Hall}) to prove that an appropriate matching can be found in (iii), we need to know, in particular, that the endpoints of each edge in $M_{23}$ have many common neighbours in the remaining part of $V_1$ and that each vertex in $V_1$ is incident to both endpoints of many edges in $M_{23}$ (and that similar conditions hold for other choices of indices). Therefore, it would be convenient if $M_{23}$ consisted only of good edges and $V_1$ contained only good vertices (see Section~\ref{subsection_good}). Unfortunately, super-regular triples can generally contain vertices that are not good. This is the reason why in step 3, we need to break down each super-regular triple into a triangle packing and a strong-super-regular triple.

Therefore, we will perform the above described process twice. First, in step 3, we will absorb all the non-good vertices into a small triangle packing by performing (i) and (iii), see Theorem~\ref{thm_decomposition}. In step 4, once we are left with a balanced strong-super-regular triple (after deleting at most $O(p^{-1})$ further vertices), we can finally perform (i)--(iii), now using only good edges to construct $M_{ij}$s, to find a triangle-factor inside this triple, see Theorem~\ref{thm_triangleblowup}.

\section{Properties of Random Graphs}
\label{section_prop-rand-graphs}

In this section we establish several properties of the random graph that
will be useful in later sections.

\begin{PROP} \label{prop_randomgraphproperties1}
For every positive real $\rho$, there exists a constant $C(\rho)$
such that if $p \ge C(\log n/n)^{1/2}$, then $G(n,p)$ a.a.s.~satisfies the
following properties.
\begin{enumerate}[(i)]
    \setlength{\itemsep}{1pt} \setlength{\parskip}{0pt}
  \setlength{\parsep}{0pt}
  \item Every vertex has degree $(1 \pm \rho)np$.
  \item Every pair of distinct vertices has $(1 \pm \rho)np^2$ common neighbours.
  \item For all $X,Y \subset V$ with $|X|,|Y| \ge \rho np$, we have $e(X,Y) = (1 \pm \rho)|X||Y|p$. In particular,
  $e(X) = e(X,X)/2 = (1 \pm \rho)|X|^2p/2$ for all $X$ of size at least $\rho np$.
\end{enumerate}
\end{PROP}

\begin{PROP} \label{prop_nontypicalvertices}
  For every $\rho \in (0,1/2)$, $G(n,p)$ satisfies the following.
  \begin{enumerate}[(i)]
  \item
    Let $D$ be a positive real. For a fixed set $W \subset V$, with probability $1 - e^{(1-\rho^2D/12)n}$,
    all but at most $nDp^{-1}/|W|$ vertices in $V \setminus W$ satisfy
    \[
    \deg(v, W) = (1 \pm \rho) |W|p.
    \]
  \item
    For every positive real $\xi$, there exists a constant $D(\rho, \xi)$ such that a.a.s.~the following holds.
    For all $W \subset V$ with $|W| \ge \xi n$, all but at most $Dp^{-1}$ vertices in $V \setminus W$ satisfy
    \[
    \deg(v, W) = (1 \pm \rho) |W|p.
    \]
  \end{enumerate}
\end{PROP}
\begin{proof}
  To prove (i), as a first step, we fix a set $W \subset V$. We may assume that $Dp^{-1}n/|W| \le n$ as otherwise, the claim is vacuously true. Suppose that there are $Dp^{-1}n/|W|$ vertices $v \in V \setminus W$ such that $\deg(v,W) \neq (1 \pm \rho)|W|p$. Then there exists a set $B \subset V \setminus W$ of size $Dp^{-1}n/(2|W|)$ such that either $\deg(v,W) > (1+\rho)|W|p$ for all $v \in B$ or $\deg(v,W) < (1-\rho)|W|p$ for all $v \in B$. This clearly implies that $e(B,W) \neq (1 \pm \rho)|B||W|p$ for some $B$ as above. Since $e(B, W)$ is a sum of independent binomial random variables and $\Ex[e(B,W)] = |B||W|p = Dn/2$, by Chernoff's inequality,
  \[
  P\big( e(B, W) \neq (1 \pm \rho)|B||W|p \big) \le e^{-\rho^2Dn/12}.
  \]
  By the union bound, the probability that such a set $B$ exists is at most $2^ne^{-\rho^2Dn/12}$.

  Now that (i) is proved, we easily get (ii) by applying the union bound.
\end{proof}

\begin{PROP} \label{prop_smallexpansion}
For all $\xi \in (0,1)$, there exists a $C(\xi)$ such that if
$p \geq C(\log n/n)^{1/2}$, then a.a.s.~for every $x \in [1, \xi
n/2]$, $G(n,p)$ does not contain a set $W$ of $x$ vertices and a set
$E$ of $x$ independent edges outside $W$ (i.e., no edge in $E$ has
an endpoint in $W$) such that either the endpoints of each edge in
$E$ have at least $\xi np^2$ common neighbours in $W$ or each vertex
in $W$ is adjacent to both endpoints of at least $\xi np^2$ edges in
$E$.
\end{PROP}
\begin{proof}
  Fix $x$, $W$, and $E$ as in the statement of this proposition. For a vertex $w \in W$ and an edge $\{u,v\} \in E$, let $B(u,v,w)$ denote the event that $w$ is adjacent to both $u$ and $v$. Let $X$ be the random variable denoting the number of events $B(u,v,w)$ that occur in $G(n,p)$. Note that each of the  ``bad'' events described in the statement of this lemma implies that $X \geq \xi np^2x \geq 2x^2p^2$. Moreover, observe that $X \leq x^2$, so we can restrict our attention to the case $x \geq \xi np^2$. Since all $B(u,v,w)$ are mutually independent, $X$ has binomial distribution with parameters $x^2$ and $p^2$, and hence by Chernoff's inequality,
  \[
  P(X \geq \xi np^2x) \leq e^{-c\xi np^2x}
  \]
  for some absolute positive constant $c$. Since for each $x$, there are at most ${n \choose x}n^{2x}$ pairs $(W,E)$ with $|W|=|E|=x$, the probability that some ``bad'' event occurs is at most
  \[
  \sum_{x = \xi np^2}^{\xi n/2} {n \choose x}n^{2x}e^{-c\xi np^2x}.
  \]
  Finally, note that
  \[
  {n \choose x}n^{2x}e^{-c\xi np^2x} \leq e^{3x\log n - c\xi np^2x} \leq n^{-2},
  \]
  provided that $np^2 \geq (4/c\xi)\log n$.
\end{proof}

\begin{PROP} \label{prop_randomgraphproperties2}
Let $p \gg n^{-1/2}$.
For every positive reals $\varepsilon$ and $\rho$,
there exists a positive real $D(\varepsilon, \rho)$
such that $G(n,p)$ a.a.s.~satisfies the
following property. For every set $X$ with $|X| \ge Dp^{-2}$, there are at most $\varepsilon n^2p$ edges $\{v, w\}$ in
$G[V \setminus X]$ such that $v$ and $w$ do not have $(1 \pm \rho)|X|p^2$ common neighbours in $X$.
\end{PROP}
\begin{proof}
The constant $D(\varepsilon, \rho)$ will be chosen later. Let $X$ be a fixed set of size at least $Dp^{-2}$. Without loss of generality we may assume that $\rho \leq 1/2$.

First expose the edges between $X$ and $V
\setminus X$ and call a pair of vertices $\{v,w\} \in V \setminus X$
\textit{bad} if $v$ and $w$ do not have $(1 \pm \rho)|X|p^2$ common neighbours in $X$.
By Proposition \ref{prop_nontypicalvertices} (i), with probability
$1 - e^{(1 - \rho^2D^{1/2}/50)n}$, there are
at most $D^{1/2}p^{-1}n/|X|$ vertices that do not satisfy $\deg(v, X) = (1 \pm \rho/2)|X|p$.
Even if each of these vertices forms bad pairs with all $n$ vertices,
there are at most $D^{1/2}p^{-1}n^2/|X|$ such bad pairs.
For each vertex that satisfies $\deg(v, X) = (1 \pm \rho/2)|X|p$, again
by Proposition \ref{prop_nontypicalvertices} (i), with probability
$1 - e^{(1 - \rho^2D^{1/2}/50)n}$, there are at most
$2D^{1/2}p^{-1}n/(|X|p)$ other vertices $w$ which do not have $(1 \pm \rho)|X|p^2$
common neighbours with $w$ in $X$. Therefore, if $D$ is sufficiently large, then
with probability at least $1 - e^{-2n}$, the total number of bad
pairs is at most
\[
\frac{D^{1/2}p^{-1}n^2}{|X|} + \frac{2D^{1/2}p^{-1}n^2}{|X|p} \le 3D^{-1/2}n^2 \le \varepsilon n^2 /2.
\]
Finally, expose the edges within $V \setminus X$. By Chernoff's inequality, with probability
$1 - e^{-\varepsilon n^2p/20}$, at most $\varepsilon n^2p$ bad pairs will form an edge.

Since $n^2p \gg n$, if we fix the set $X$, both of the above events happen with probability
at least $1 - e^{-2n}$. Since there are at most $2^n$ choices for $X$, we can take
the union bound over all choices of $X$ to derive the conclusion.
\end{proof}

Using the above propositions, we now prove the following generalization of \cite[Lemma 6.4]{HuLeSu}.

\begin{LEMMA} \label{lemma_leftovertriangles}
There exist $C$, $D$, and $\varepsilon$
such that if $p \ge C(\log n/n)^{1/2}$, then $G(n,p)$ a.a.s.~has the following property. For every
spanning subgraph $G' \subset G(n,p)$ with $\delta(G') \geq (2/3)np$ and every set
$T \subset V(G')$ with $|T| \leq \varepsilon n$, all but at most $Dp^{-2}$ vertices of $V \backslash
T$ are contained in a triangle of $G$ which does not intersect $T$.
\end{LEMMA}
\begin{proof}
For the sake of brevity, denote $G(n,p)$ by $G$ and let $V = V(G)$. Let $\varepsilon$ be a small positive constant (we will fix it later), let $C = C_{\ref{prop_randomgraphproperties1}}(\varepsilon)$, and let $D_0(\varepsilon)$ be a constant satisfying $D_0 \ge \max
\{D_{\ref{prop_nontypicalvertices} (ii)}(\varepsilon, \varepsilon), D_{\ref{prop_randomgraphproperties2}}(\varepsilon, \varepsilon), 1\}$.

Without loss of generality we may assume that $|T| = \varepsilon
n$. Let $X_0 \subset V \setminus T$ be an arbitrary set of size
$2D_0p^{-2}$. By assuming that the events from Propositions
\ref{prop_randomgraphproperties1}, \ref{prop_nontypicalvertices}~(ii), and \ref{prop_randomgraphproperties2} hold, we will show that
there exists a triangle in $G'$ which intersects $X_0$ but not $T$.
This will prove that there are at most $2D_0p^{-2}$ vertices that are not
contained in triangles that do not hit $T$. Let $T'$ be the collection of all the vertices $v$
that satisfy $\deg(v, X_0 \cup T) \ge (1 + \varepsilon)|X_0 \cup T|p$ and note that
$|T'| \le D_0p^{-1}$ by Proposition~\ref{prop_nontypicalvertices}~(ii). Let $T'' = T \cup T'$ and $X =
X_0 \setminus T'$. Note that $|T''| \le 2\varepsilon n$ and $|X| \ge
|X_0| - |T'| \ge D_0 p^{-2}$. Let $D$ be the constant defined by $|X| = Dp^{-2}$ and note that
$D \ge D_0$. It suffices to show that there exists a triangle in $G'$ which
contains a vertex from $X$ but not from $T''$.

Let $Y = V \setminus (X \cup T'')$ and fix a vertex $x \in X$. Note that
\begin{align*}
  \deg_{G'}(x,Y) & = \deg_{G'}(x) - \deg_{G'}(x, X \cup T'') \geq \deg_{G'}(x) - \deg_{G'}(x, X_0 \cup T) - |T'| \\
  & \geq (2/3)np - (1+\varepsilon)|T \cup X_0|p - D_0p^{-1} \geq (2/3-3\varepsilon)np,
\end{align*}
where the last two inequalities follow from the fact that $x \not\in T'$ and our assumption on $p$. Finally, let $N_x = N_G(x) \cap Y$ and fix an arbitrary subset $N_x' \subset N_{G'}(x) \cap Y$ of size $(2/3 - 3\varepsilon)np$.

It suffices to show that the number of triangles $xy_1y_2$ in
$G'$ such that $x \in X$ and $y_1, y_2 \in N_x'$ is nonzero. Let this
number be $M$. To bound $M$ from below, first bound the number of
triangles $xy_1y_2$ in $G$ such that $x \in X, y_1, y_2 \in N_x$,
and $y_1y_2$ is an edge of the graph $G'$ (we will later subtract
the number triangles whose $y_1$ or $y_2$ is not in $N_x'$). Let this
number be $M_0$. Since $|X \cup T''|, |Y| \geq \varepsilon n$, by
Proposition~\ref{prop_randomgraphproperties1}~(iii), we have
\[ e_{G'}(Y) \geq e_{G'}(V) - e_{G'}(Y, X \cup T'') - e_{G'}(X \cup T'') \geq \frac{2}{3} \cdot \frac{n^2p}{2} - 4\varepsilon n^2p =\left(\frac{2}{3} - 8\varepsilon \right)\frac{n^2p}{2}.\]
By Proposition \ref{prop_randomgraphproperties2}, the number of
edges $\{v, w\}$ in $G'[Y]$ that form a triangle in $G$ with fewer than
$(1 - \varepsilon)D$ vertices in $X$ is at most $\varepsilon n^2p$
given that $D_0$ is large enough. Thus,
\[
M_0 \ge \Big(e_{G'}(Y) - \varepsilon n^2p \Big) (1 - \varepsilon)D \ge \left(\frac{2}{3} - 11\varepsilon \right)\frac{Dn^2p}{2}.
\]
To obtain a bound on $M$ from $M_0$, we can subtract the number of
triangles $xy_1y_2$ as above such that either $y_1$ or $y_2$ is not
in $N_x'$. Since $|N_x| \leq (1 + \varepsilon)np$ by
Proposition~\ref{prop_randomgraphproperties1}~(i), we have
\[
|N_x \setminus N_x'| = |N_x| - |N_x'| \leq (1 + \varepsilon)np - (2/3 - 3\varepsilon)np = (1/3 + 4\varepsilon)np.
\]
Thus, if $\varepsilon$ is small enough, by Proposition
\ref{prop_randomgraphproperties1}~(iii) we have,
\begin{align*}
 M & \ge M_0 - \sum_{x \in X} \left( e_{G'}(N_x \setminus N_x', N_x') + e_{G'}(N_x \setminus N_x')\right) \\
    &\ge M_0 - \sum_{x \in X} \left(1+\varepsilon\right)\left( \left(\frac{1}{3} + 4\varepsilon\right) \frac{2}{3} n^2p^3 + \left(\frac{1}{3} + 4\varepsilon\right)^2\frac{n^2p^3}{2} \right) \\
   &\ge \left(\frac{2}{3} - 11\varepsilon \right)\frac{Dn^2p}{2} - \sum_{x \in X} \left( \frac{5}{9} + 10\varepsilon \right)\frac{n^2p^3}{2} = \left(\frac{1}{9} - 21\varepsilon \right)\frac{Dn^2p}{2}.
\end{align*}
Therefore there exists a triangle as claimed, provided that $\varepsilon$ is sufficiently small.
\end{proof}

The following proposition establishes the fact that it is necessary to have $\Omega(p^{-2})$ vertices not covered by triangles. Its proof closely follows the argument from~\cite[Proposition 6.3]{HuLeSu}.

\begin{PROP} \label{prop_removetriangles}
  Let $\varepsilon > 0$. There exists a positive constant $C(\varepsilon)$ such that if $Cn^{-1/2} \le p \ll 1$, then $G(n,p)$ a.a.s.~contains a spanning subgraph of minimum degree at least $(1-\varepsilon)np$ such that $\Omega(p^{-2})$ of its vertices are not contained in a triangle.
\end{PROP}
\begin{proof}
  Let $C$ be a constant satisfying $C \geq 2$ and $e^{C^2/15} \geq 8e/\varepsilon$. If $p \geq (\log n/n)^{1/2}$, then by Proposition~\ref{prop_randomgraphproperties1}, a.a.s.~$\delta(G(n,p)) \geq (1-\varepsilon/4)np$ and each pair of vertices of $G(n,p)$ has at most $2np^2$ common neighbours. If $Cn^{-1/2} \leq p < (\log n/n)^{1/2}$, then still a.a.s.~$\delta(G(n,p)) \geq (1-\varepsilon/4)np$, but $G(n,p)$ may contain some edges whose endpoints have more than $2np^2$ common neighbours. Let $H$ be the subgraph consisting of all such edges, and let $v$ be an arbitrary vertex. By Chernoff's inequality, the probability that $v$ and some other vertex have more than $2np^2$ common neighbours is at most $e^{-np^2/15}$. Therefore,
  \[
  P(\deg_H(v) \geq (\varepsilon/4)np ) \leq {n \choose (\varepsilon/4)np}\left(p \cdot e^{-np^2/15}\right)^{(\varepsilon/4)np} \leq \left(\frac{4enp}{\varepsilon np} \cdot e^{-C^2/15}\right)^{(\varepsilon/4)np} = o(n^{-1}),
  \]
  so a.a.s.~$\Delta(H) < (\varepsilon/4)np$. Finally, let $G = G(n,p) - H$. Clearly, the endpoints of every edge of $G$ have at most $2np^2$ common neighbours. Moreover, by Proposition~\ref{prop_randomgraphproperties1}~(i), we may assume that $\delta(G) > (1-\varepsilon/2)np$.

 Let $X$ be an arbitrary fixed set of $(\varepsilon/4)p^{-2}$ vertices of $G$ and let $W = \{v \notin X \colon \deg(v,X) \geq 2|X|p\}$. By Chernoff's inequality, the probability that a vertex $v$ belongs to $W$ is $e^{-\Omega(p^{-1})}$ and these events are independent for different vertices. Since $p \gg e^{-\Omega(p^{-1})}$, Chernoff's inequality implies that a.a.s.~$|W| \leq (\varepsilon/4)np$. Moreover, since our assumption on $p$ implies that $|X| \leq (\varepsilon/8) n$, we can apply Chernoff's inequality
 and deduce that a.a.s.~$\deg(u,X) \leq (\varepsilon/4) np$ for every vertex $u$.

Let $G'$ be the subgraph of $G$ obtained by deleting all edges within $X$, all edges between $X$ and $W$, and deleting edges incident to any $y \notin X \cup W$ according to the following rule -- for every triangle $xyz$ in $G$ with $x \in X$ and $z \notin X \cup W$, remove the edge $yz$. It is quite easy to see that no vertex of $X$ is contained in a triangle in $G'$. Let us now estimate $\delta(G')$. Since a vertex $u \in X \cup W$ lost only edges connecting it to $X$ and $W$, we have
\[
\deg_{G'}(u) > (1-\varepsilon/2)np - \deg(u,X) - \deg(u,W) \geq (1-\varepsilon/2)np - \deg(u,X) - |W| \geq (1-\varepsilon)np.
\]
Since a vertex $y \notin X \cup W$ is incident to at most $(\varepsilon/2)p^{-1}$ vertices $x \in X$ and it has at most $2np^2$ common neighbours with each such $x$, we then have
\[
\deg_{G'}(y) > (1-\varepsilon/2)np - (\varepsilon/4)p^{-1} \cdot 2np^2 \geq (1-\varepsilon)np.
\]
Thus $G'$ has the required properties.
\end{proof}

We end this section with two propositions whose proofs are farily standard and
are omitted.
Proposition \ref{prop_reducedmindegree} asserts that in a typical random graph
$G(n,p)$, the reduced graph of a regular partition of a subgraph $G \subset G(n,p)$
inherits the minimum degree condition that we impose on $G$. The final proposition,
Proposition \ref{prop_randomgraphtypical} can be proved using
Lemma~\ref{lemma_typicalregulartriple}, and asserts that every
regular triple in a random graph is typical.

\begin{PROP} \label{prop_reducedmindegree}
Let $\gamma>0$ and $p \gg n^{-1}$. There exist $\varepsilon_0(\gamma)$ and
$\delta_0(\gamma)$ such that if $\varepsilon \le \varepsilon_0$ and $\delta \le \delta_0$,
then the following holds asymptotically almost surely. Given a subgraph $G$ of $G(n,p)$, let $(V_i)_{i=1}^{k}$ be an
$(\varepsilon,p)$-regular partition of $G$ such that $|V_i| \le \varepsilon
n$ for all $i$ and every part forms an $(\varepsilon,p)$-regular
pair with at least $(1 - \varepsilon)k$ other parts. Let $R$ be its
$(\delta, \varepsilon, p)$-reduced graph. If $G$ has minimum degree
at least $(2/3 + \gamma)np$, then $R$ has minimum degree at least
$(2/3 + \gamma/2)k$.
\end{PROP}


\begin{PROP} \label{prop_randomgraphtypical}
Let $p \gg n^{-1/2}$. For all positive $\varepsilon'$, $\delta$, and $\xi$,
there exists a constant $\varepsilon_0(\varepsilon',\delta)$ such that a.a.s.~in $G(n,p)$,
every copy of a graph from $\FAMILY{K_3}{(n_1, n_2, n_3)}{(d_{12},
d_{23}, d_{31})}{(\varepsilon,p)}$ is $\varepsilon'$-typical provided that
$\varepsilon \le \varepsilon_0$, $d_{12}, d_{23}, d_{31}
\ge \delta$, and $n_1, n_2, n_3 \ge \xi n$.
\end{PROP}
\section{Obtaining balanced super-regular triples}
\label{section_cleanup1}

In Section~\ref{subsection_outline2}, we mentioned that the process of absorbing exceptional vertices into a triangle packing may cause some regular triples in our graph to become slightly unbalanced. The following lemma describes a greedy procedure that finds a small triangle packing which restores the balance in each of these triples.

\begin{LEMMA} \label{lemma_balancingpartition}
Let $p \gg n^{-1/2}$. For all positive reals $\delta$,
$\varepsilon'$, and $\gamma$, there exists
an $\varepsilon_0(\delta, \varepsilon',\gamma)$ such that
if $\varepsilon < \varepsilon_0$, then the following holds asymptotically almost surely. Let $G$
be a subgraph of $G(n,p)$ and let $V_1, \ldots, V_{3k}$ be disjoint
subsets of $V(G)$ satisfying $|V_i| \in [(1 - \varepsilon)m,
(1+\varepsilon)m]$ for some $m = \Omega(n)$. Let $R$ be a graph on the
vertex set $[3k]$ of minimum degree at least $(2 + \gamma)k$ such
that $\{3t-2, 3t-1, 3t\}$ forms a triangle for all $t \in [k]$ and
assume that $(V_i)_{i=1}^{3k}$ is $(\delta, \varepsilon,p)$-regular
over $R$.

Then there exist subsets $B$ and $S$ of $V(G)$ such that the
following holds.
\begin{enumerate}[(i)]
    \setlength{\itemsep}{1pt} \setlength{\parskip}{0pt}
    \setlength{\parsep}{0pt}
\item $|B| \le 4k$,
\item $G[S]$ contains a perfect triangle packing,
\item $|V_i \cap (B \cup S)| \le \varepsilon' m$ for all $i \in
[3k]$, and
\item $V_i \setminus (B \cup S)$ have equal sizes for all $i$.
\end{enumerate}
\end{LEMMA}
\begin{proof}
Let $\varepsilon_1 = \varepsilon_{\ref{prop_findtriangle}}(\frac{1}{2},\delta)$
and $\varepsilon_0 = \min \{ \frac{\varepsilon'}{4(3/\gamma + 1)},
\varepsilon_{\ref{prop_randomgraphtypical}}(\varepsilon_1,\frac{\delta}{2}),
\frac{\delta}{2}, \varepsilon_1 \}$. Assume that $\varepsilon \in (0, \varepsilon_0)$ is given.

Let $C_t = \{3t-2, 3t-1, 3t\}$ for $t \in [k]$ be triangles of the
graph $R$. For each vertex $i \in [3k]$, call an index $t \in [k]$
{\em $i$-rich} or {\em rich with respect to $i$} if $i$ is adjacent to all three vertices of $C_t$, and
assume that there are $g_i$ $i$-rich indices. Then by the
minimum degree condition on $R$, we have
\[ 3g_i + 2(k - g_i) \ge (2 + \gamma)k, \]
which is equivalent to $g_i \ge \gamma k$. Thus for each vertex $i
\in [3k]$ of $R$, we can assign an $i$-rich index $t \in [k]$ to it so that
every index in $[k]$ is chosen by at most $(3k)/(\gamma k) =
3/\gamma$ vertices.

Consider the following process that adjusts the parts one by one.
Throughout the process, we will maintain sets $B \subset V(G)$ and
$Z_i \subset V_i$ for each $i \in [3k]$; they are empty at the beginning. Call a
triangle $C_t$ balanced if the sets $V_{3t-2} \setminus Z_{3t-2}$,
$V_{3t-1} \setminus Z_{3t-1}$, and $V_{3t} \setminus Z_{3t}$ have equal
cardinalities. Assume that the triangles $C_1, \ldots, C_{t-1}$ are
already balanced and we are trying to balance the triangle $C_t$.
Without loss of generality, we may assume that $|V_{3t-i} \setminus
Z_{3t-i}| - |V_{3t} \setminus Z_{3t}| = x_i m$ and $0 \le x_i \le
2\varepsilon$ for $i \in \{1,2\}$. Thus we have to remove $x_i m$ vertices
from $V_{3t-i}$ for $i \in \{1,2\}$ in order to make $C_t$ balanced. By
moving at most $2$ arbitrary vertices from each set $V_{3t-1}$ and
$V_{3t-2}$ to $B$ and also to $Z_{3t-1}$ and $Z_{3t-2}$, respectively,
we may assume that both $x_1 m$ and $x_2 m$ are divisible by 3. First
consider the set $V_{3t-1}$ and let $s$ be the rich index
with respect to $3t-1$ which we have chosen above. As we will later establish, for every
$i \in [3k]$, $|V_{i} \setminus Z_i| \ge m/2$ throughout the
process. Therefore by Proposition~\ref{prop_inheritregularity}, the
triple $(V_{3t-1} \setminus Z_{3t-1}, V_{3s-j} \setminus Z_{3s-j},
V_{3s-k} \setminus Z_{3s-k})$ inherits the regularity of $(V_{3t-1},
V_{3s-j}, V_{3s-k})$ and is always $(2\varepsilon, p)$-regular of
density at least $\delta/2$ for every pair
$\{j, k\} \subset \{0,1,2\}$. By Proposition~\ref{prop_randomgraphtypical}, a.a.s.~it
must also be $\varepsilon_1$-typical. Thus by
Proposition~\ref{prop_findtriangle} we can find $x_1 n /3$ triangles across
this triple. Do this for each pair $\{j, k\}$
and update the sets $Z_{3t-1}$, $Z_{3s-2}$, $Z_{3s-1}$, and $Z_{3s}$ by
placing all the vertices of these triangles into corresponding parts.
Note that even though the sizes of the sets in $C_s$ have decreased,
the number by which they decreased is the same for all three of them
and thus after performing the same procedure for $V_{3t-2}$, the triangles
$C_1, \ldots, C_t$ will be balanced.

Note that in the end, $|B| \le 4k $. Moreover, throughout the
process, by the restriction that every index is the chosen rich index for at
most $3/\gamma$ other indices, we always have, $|Z_i| \le
2\varepsilon m \cdot (3/\gamma + 1) \le \min\{\varepsilon' m, m/2\}$
as claimed. Define $S = \big( \cup_{i=1}^{3k} Z_i \big) \setminus B$
and we have the sets $B$ and $S$ as claimed.
\end{proof}

Below is the main theorem of this section. It says that we can partition our graph into balanced super-regular triples, a collection of vertex-disjoint triangles, and a set of at most $O(p^{-2})$ exceptional vertices. We would like to remark that the upper bound imposed on the sizes of the common neighbourhoods in (v) will come in handy in the proof of Theorem~\ref{thm_decomposition}, where we show that the triples $(W_{3t-2}, W_{3t-1}, W_{3t})$ are close to being strong-super-regular, see Proposition~\ref{prop_manygoodvertices}.

\begin{THM} \label{thm_cleanstage1}
For an arbitrary $\gamma$, there exist
$\delta(\gamma)$ and $\varepsilon_0(\gamma)$
such that for all $\varepsilon \in (0,\varepsilon_0)$, there exist
constants $C(\varepsilon), D(\varepsilon)$, and $\xi(\varepsilon)$
satisfying the following. If $p \ge C(\log n / n)^{1/2}$,
then a.a.s.~for every spanning subgraph $G' \subset
G(n,p)$ with $\delta(G') \ge (2/3 + \gamma)np$, there exist a
further subgraph $G'' \subset G'$ and a partition of $V(G)$ into
sets $B$, $S$, and $(W_i)_{i=1}^{3k}$, where $k \leq D$, such that
\begin{enumerate}[(i)]
    \setlength{\itemsep}{1pt}
    \setlength{\parskip}{0pt}
    \setlength{\parsep}{0pt}
\item $|B| \le Dp^{-2}$.
\item $G'[S]$ contains a perfect triangle packing.
\item $(W_{3t-2}, W_{3t-1}, W_{3t})$ is a $(\delta, \varepsilon,p)$-super-regular triple in $G''$ for all $t \in [k]$.
\item $|W_{3t-2}| = |W_{3t-1}| = |W_{3t}| \ge \xi n$ for all $t \in [k]$.
\item In the graph $G''$, for all $t \in [k]$, the endpoints of every edge in $(W_{3t-2}, W_{3t-1})$ have at most
$4|W_{3t}|p^2$ common neighbours in $W_{3t}$ and a similar statement holds for other choices of indices.
\end{enumerate}
\end{THM}
\begin{proof}
Given a $\gamma$, let $\delta =
\delta_{\ref{prop_reducedmindegree}}(\gamma)$ and $\varepsilon_0 =
\varepsilon_{\ref{prop_reducedmindegree}}(\gamma)$. Moreover, for a given
$\varepsilon \in (0,\varepsilon_0)$, let $\varepsilon_1 \le \min \{
\varepsilon/2,
\varepsilon_{\ref{prop_typicalstability}}(\varepsilon, \delta),
\delta/2\}$, $\varepsilon_2 \le \min \{
\varepsilon_{\ref{lemma_leftovertriangles}}^2/36^2,
\varepsilon_1^2/400,
\varepsilon_{\ref{lemma_balancingpartition}}(\delta,
\varepsilon_1/3, \gamma/2) \}$, $\varepsilon_3 \le (1/27)\min\{
\varepsilon_{\ref{prop_randomgraphtypical}}(\varepsilon_2,
\delta)^3, \varepsilon_2\}$.

Let $K = K_{\ref{thm_regularity}}(\varepsilon_3, 2,
1/\varepsilon_3)$, $\eta =
\min \{\eta_{\ref{thm_regularity}}(\varepsilon_3, 2, 1/\varepsilon_3), 1\}$,
$\xi = \varepsilon_3/(4K)$, $C = \max\{C_{\ref{prop_randomgraphproperties1}}(\eta), C_{\ref{lemma_leftovertriangles}} \}$, and $D =
\max\{3D_{\ref{lemma_leftovertriangles}},
18KD_{\ref{prop_nontypicalvertices} (ii)}(1/3,\xi), 24K \}$.

Proposition~\ref{prop_randomgraphproperties1}~(iii) implies that
$G(n,p)$ is a.a.s.~$(\eta, 2, p)$-upper-uniform. Thus we can apply
the regularity lemma, Theorem \ref{thm_regularity}, to obtain an
$(\varepsilon_3,p)$-regular partition $V_0, V_1, \ldots, V_{3k}$ of
the graph $G'$, where each part forms a regular pair with at least
$(3-3\varepsilon_3)k$ other parts. Let $m = |V_i|$, note that $m \ge n/(2K)$, and let $R$ be the
reduced graph with parameter $\delta$. Since $G'$
has minimum degree at least $(2/3 + \gamma)np$, by Proposition~\ref{prop_reducedmindegree},
a.a.s.~the reduced graph has minimum degree at
least $(2 + \gamma/2)k$. Thus by Theorem~\ref{thm_CoHa}, we may
assume that $(V_{3t-2}, V_{3t-1}, V_{3t})$ forms an
$(\varepsilon_3,p)$-regular triple of density at least $\delta$ for
all $t \in [k]$. By Proposition~\ref{prop_randomgraphproperties2},
a.a.s.~there are at most $(\varepsilon_3/2) pm^2$ edges in $(V_{3t-1},
V_{3t-2})$ whose endpoints have more than $2|V_{3t}|p^2$ common
neighbours in $V_{3t}$. Similar estimate holds for the edges in
$(V_{3t-2}, V_{3t})$ and $(V_{3t-1}, V_{3t})$. Delete all such edges
for all $t \in [k]$ to obtain the subgraph $G''$. Then in the graph
$G''$, each triple $(V_{3t-2}, V_{3t-1}, V_{3t})$ is
$(3\varepsilon_3^{1/3}, p)$-regular by Proposition~\ref{prop_inheritregularity_edgever}.

By Proposition~\ref{prop_randomgraphtypical}, we may assume that every
$(3\varepsilon_3^{1/3},p)$-regular triple of density at least
$\delta$ is $\varepsilon_2$-typical. Thus for each index $i$, if
we let $X_i \subset V_i$ be the collection of non
$\varepsilon_2$-typical vertices, then $|X_i| \le \varepsilon_2
|V_i|$. Furthermore, for each $t \in [k]$, add to $X_{3t}$ the collection of those vertices
$v \in V_{3t}$ such that $\deg(v, V_{3t-j}) \neq (1 \pm
\varepsilon_2) d_{3t,3t-j}p|V_{3t-j}|$ for some $j \in \{1,2\}$ and define
$X_{3t-1}$ and $X_{3t-2}$ accordingly (there are at most $4\varepsilon_2 n$
such vertices by regularity). By adding arbitrary vertices to $X_i$
if necessary, we may assume that $|X_i| = 5\varepsilon_2 |V_i|$.
Move all the vertices in $X_i$ from $V_i$ to $V_0$ and denote the
resulting partition by $(V_i')_{i=0}^{3k}$. We then have $|V_0'| \le
|V_0| + \sum_{i}|X_i| \le 6\varepsilon_2 n$.

Consider the following process of finding triangles that absorbs the
vertices in $V_0'$. Let $T$ be the empty set; we will update it throughout
the process. Apply Lemma \ref{lemma_leftovertriangles} to find a
triangle which hits $V_0'$ but not $T$ and move all the vertices of
this triangle into $T$. If $|T \cap V_i'| \ge \sqrt{\varepsilon_2}
|V_i'|- 3$ for some index $i \in [3k]$, then move all the vertices of
$V_i'$ into $T$. This way, we will have
\[
|T| \le 3|V_0'|\cdot(1/\sqrt{\varepsilon_2}) + 3|V_0'| \le 36\sqrt{\varepsilon_2}n \le \varepsilon_{\ref{lemma_leftovertriangles}}n
\]
throughout the process. Terminate the process when we cannot find
such triangles anymore. Then, a.a.s.~we must have $|V_0'| \le (D/3)p^{-2}$. Let
$B_0$ be the collection of all the remaining vertices of $V_0'$, and
$S_0$ be the set of vertices in the copies of the triangles that we found.
Let $V_i'' = V_i' \setminus S_0 = V_i \setminus (B_0 \cup S_0)$ and note that
for all $i$, since $|V_i \cap (B_0 \cup S_0)| \le |X_i| + \sqrt{\varepsilon_2}|V_i| \le
(\varepsilon_1/6)|V_i|$, then $|V_i''| \ge (1 - \varepsilon_1/6)|V_i|$.

By Proposition \ref{prop_inheritregularity}, $(V_{3t-2}'',
V_{3t-1}'', V_{3t}'')$ forms a $(2\varepsilon_3,p)$-regular triple
of density at least $\delta - \varepsilon_3 \ge \delta/2$ for all $t
\in [k]$. Apply Lemma \ref{lemma_balancingpartition} to
$(V_i'')_{i=1}^{3k}$ to obtain sets $B_1$ and $S_1$. Observe that
$|B_0 \cup B_1| \le (D/3)p^{-2} + 4K \le (2D/3)p^{-2}$ and $G[S_0 \cup S_1]$ contains a perfect triangle
packing. Also, most crucially, if we let $BS = B_0 \cup S_0 \cup B_1 \cup
S_1$ and $W_i = V_i \setminus BS = V_i'' \setminus (B_1 \cup
S_1)$, then all $W_i$ have equal sizes and moreover, $|W_i| \geq (1 -
\varepsilon_1/6 - \varepsilon_1/3)m \ge (1 - \varepsilon_1/2)m$.

We will remove some vertices from each set $W_i$ to make the triples
$(W_{3t-2}, W_{3t-1}, W_{3t})$ super-regular for all $t \in [k]$.
Since $\xi n \le |X_i| \le |BS \cap V_i| \le (\varepsilon_1/2) m$
for all $i$, by Proposition \ref{prop_nontypicalvertices}~(ii),
there are at most $(D/(18K))p^{-1}$ vertices which have more than
$(2\varepsilon_1/3) p m$ neighbours in $BS \cap V_i$ for each fixed
$i \in [3k]$. Let $Y_1$ be the collection of such vertices for the
set $BS \cap V_2$ and $BS \cap V_3$ which lie in $V_1$ and similarly
define $Y_2, Y_3$. By placing arbitrary vertices into $Y_1, Y_2$, or
$Y_3$ as necessary, we may assume that $|Y_1| = |Y_2| = |Y_3| \le
(D/(9k))p^{-1}$. Consider the set $W_i' = W_i \setminus Y_i$ for
$i \in \{1,2,3\}$. Then since $|Y_i| + |BS \cap V_i| \le (D/(9K))p^{-1} +
(\varepsilon_1/2) m \le \varepsilon_1 m$, in total we removed at
most $\varepsilon_1 m$ vertices from each part of $(V_1,V_2,V_3)$ to
obtain $(W_1',W_2',W_3')$. By the definition of the sets $X_i$ at
the beginning, all the vertices in $W_1'$ were
$\varepsilon_1$-typical in the triple $(V_1, V_2, V_3)$, and by the
choice of $Y_1$, they have at most $(2\varepsilon_1/3) p m + |Y_2|
\le \varepsilon_1 pm$ neighbours in the deleted portion in $V_2$
(similar for $V_3$). Thus by
Proposition~\ref{prop_typicalstability}, all the vertices in $W_1'$
are $\varepsilon$-typical in the triple $(W_1',W_2',W_3')$. Also,
since all the vertices of $V_1$ not in $X_1$ had $(1 \pm
\varepsilon_1)d_{12}p|V_2|$ neighbours in $V_2$, they will still have
$(1 \pm 2\varepsilon_1)d_{12}p|V_2|$ neighbours in $W_2'$ and similar
for other choices of indices. Moreover, the triple $(W_1', W_2',
W_3')$ inherits the regularity of $(V_1, V_2, V_3)$ and is
$(2\varepsilon_3, p)$-regular of density at least $\delta -
\varepsilon_3 > \delta/2$, see Proposition~\ref{prop_inheritregularity}.
Thus by the fact $2\varepsilon_1 <
\varepsilon$ and $2\varepsilon_3 < \varepsilon$, $(W_1',W_2',W_3')$
is $(\delta/2, \varepsilon,p)$-super-regular. Repeat the above process
for all other triples. Let $B$ be the union of $B_0, B_1$, and $Y_i$ for
all $i \in [3k]$ as above so that $|B| \le (2D/3)p^{-2} +
(D/3)p^{-1} = Dp^{-1}$ and $S = S_0 \cup S_1$. We also have the
bound $|W_i'| \ge |V_i|/2 \ge \xi n$ for all $i$.  Moreover, (v) will hold
since in $G''$ all the edges between $W_2'$ and $W_3'$
have at most $2|V_i|p^2$ common neighbours in $V_1$,
and therefore at most $4|W_1'|p^2$ in $W_1'$ (similar for other indices).
\end{proof}

\section{Obtaining balanced strong-super-regular triples}
\label{section_cleanup2}

In the previous section, we managed to decompose the graph into balanced super-regular triples, a triangle packing, and a small set of exceptional vertices. In this section, we will show how by slightly enlarging the triangle packing and the exceptional set, we can make these triples strong-super-regular. 

Our main tool, which will also be used in the next section, is the following lemma, which constructs small quasi-random matchings in super-regular triples. 
For the application in this section, in Theorem~\ref{thm_decomposition} below, 
$V_i'$s will be the sets of non-good vertices 
in each part of a regular partition of the host graph. We want to find vertex disjoint triangles that cover these sets of non-good vertices.
As an intermediate step, we construct random matchings $M_{ij}'$ which
later can be coupled with the non-good vertices in order to construct vertex-disjoint triangles.
See the discussion in Section~\ref{subsection_outline34} for more detailed description.
We would like to remark that even though the stronger assumption (A1) implies the weaker assumption (A2), we state both of them, as (A1) is much simpler and in one of the two applications of Lemma~\ref{lemma_randommatching}, we can verify that this stronger condition is satisfied.
Also note that the statement of this lemma holds not only for strong-super-regular triples 
coming from subgraphs of random graphs, but also for general strong-super-regular triples.

\begin{LEMMA}
  \label{lemma_randommatching}
  For all positive $\delta$ and $\eta$ with $\eta < 1/140$, there exist $\varepsilon(\delta)$ and $C(\delta, \eta)$ such that the following holds. Let $(V_1, V_2, V_3)$ be a $(\delta, \varepsilon, p)$-super-regular triple with $m = |V_1| = |V_2| = |V_3|$ and $p \geq C(\log m/m)^{1/2}$. For each $i$ and $j$, let $d_{ij}p$ be the density of $(V_i, V_j)$, let $q_{ij} = \eta/(d_{ij}pm)$, and let $E_{ij}$ be a subgraph of $(V_i, V_j)$ with $|E_{ij}| \leq \eta d_{ij}pm^2$. Form a set $M'_{ij}$ by selecting every edge in $(V_i, V_j) \setminus E_{ij}$ independently with probability $q_{ij}$ and let $M_{ij} \subset M'_{ij}$ be the set of all selected edges in $(V_i, V_j)$ that are not incident to any other edge in $M'_{12} \cup M'_{13} \cup M'_{23}$. Moreover, for each $i$, let $Q_i$ be the set of all vertices in $V_i$ that are covered by some edge in $M_{12} \cup M_{13} \cup M_{23}$. Assume that for each $i$, $j$, and $k$, there is a set $V'_i$ such that
  \begin{enumerate}
  \item[(A1)]
    the neighbourhood of every $v \in V'_i$ contains at most $\eta d_{ij}d_{ik}d_{jk}p^3m^2$ edges of $E_{jk}$ or
  \item[(A2)]
    for every $v \in V'_i$, every subgraph $(N''_j, N''_k)$ of $(N(v) \cap V_j, N(v) \cap V_k)$ such that $\deg(w, N''_k) \leq 2d_{ik}d_{jk}p^2m$ for all $w \in N''_j$ and $\deg(w, N''_j) \leq 2d_{ij}d_{jk}p^2m$ for all $w \in N''_k$ contains at most $\eta d_{ij}d_{ik}d_{jk}p^3m^2$ edges of $E_{jk}$.
  \end{enumerate}
  Then $M_{12} \cup M_{13} \cup M_{23}$ is a matching and with probability tending to $1$ as $m$ tends to infinity, for each $i$, $j$, and $k$,
  \begin{enumerate}
  \item[(M1)]
    $(\eta/2)m \leq |M_{ij}| \leq 2\eta m$,
  \item[(M2)]
    every $v \in V_i$ has at most $3\eta d_{ij}pm$ neighbours in $Q_j$,
  \item[(M3)]
    the neighbourhood of every $v \in V'_i$ contains at least $(\eta/2)\delta^2p^2m$ edges of $M_{jk}$, and
  \item[(M4)]
    the endpoints of each $\eta$-good edge in $(V_j, V_k)$ have at least $(1-4\eta)d_{ij}d_{ik}p^2m$ common neighbours in $V_i \setminus Q_i$.
  \end{enumerate}
\end{LEMMA}
\begin{proof}
 Let $\varepsilon = \min\{1/100, \delta/50\}$. Fix $i$, $j$, and $k$ with $\{i,j,k\} = \{1,2,3\}$. For 
 each vertex $v \in V_i$, let $N_j = N(v) \cap V_j$ and $N_k = N(v) \cap V_k$. By construction, $M_{12} \cup M_{13} \cup M_{23}$ is a matching.

  \begin{CLAIM}
    \label{claim_M1}
    With probability $1 - o(1)$, $(\eta/2)m \leq|M_{ij}| \le 2\eta m$.
  \end{CLAIM}
  \begin{proof}
    By our assumption on $|E_{ij}|$, there are at least $(1-\eta)d_{ij}pm^2$ edges in $(V_i, V_j) \setminus E_{ij}$, so
    $\BBE[|M'_{ij}|] \ge (1 - \eta)d_{ij}pm^2q_{ij} = (1 - \eta)\eta m$, and Chernoff's inequality implies that $|M'_{ij}| \geq (3\eta/4)m$ with probability $1 - o(1)$. In order to estimate $|M_{ij}|$, note that $|M'_{ij}| - |M_{ij}|$ is at most the number of vertices in $V_i \cup V_j$ that are incident to an edge of $M'_{ij}$ and some other edge in $M'_{12} \cup M'_{13} \cup M'_{23}$. Let $\mathcal{A}_w$ denote the event that $w$ is such a ``bad'' vertex. Since $(V_1, V_2, V_3)$ is $(\delta, \varepsilon, p)$-super-regular, $\deg(v, V_j) \leq (1+\varepsilon)d_{ij}pm$ and $\deg(v, V_k) \leq (1+\varepsilon)d_{ik}pm$ for every $v \in V_i$. Hence, if $w \in V_i$, then
    \[
    P(\mathcal{A}_w) \leq \deg(w, V_j)q_{ij} \cdot (\deg(w,V_j)q_{ij} + \deg(w,V_k)q_{ik}) \leq (1+\varepsilon)^22\eta^2
    \]
    and the expected number of such ``bad'' vertices in $V_i$ is at most $(1+\varepsilon)^22\eta^2m$.
    The events $\{\mathcal{A}_w \colon w \in V_i\}$ are mutually independent, so by Chernoff's inequality, with probability at least $1 - o(m^{-1})$, there are at most $3\eta^2m$ ``bad'' vertices in $V_i$ and similarly, there are at most $3\eta^2m$ ``bad'' vertices in $V_j$. Hence, $|M_{ij}| \geq (3/4 - 6\eta)\eta m \geq (\eta/2)m$ with probability $1 - o(1)$. Finally, since the number of edges in $(V_i, V_j) \setminus E_{ij}$ is at most $d_{ij}pm^2$, we have $\BBE[|M'_{ij}|] \le d_{ij}pm^2q_{ij} = \eta m$, and Chernoff's inequality implies that $|M_{ij}| \leq |M'_{ij}| \leq 2\eta m$ with probability $1 - o(1)$.
  \end{proof}

  \begin{CLAIM}
    \label{claim_M2}
    For each fixed vertex $v$, with probability $1 - o(m^{-1})$, we have $\deg(v, Q_j) \leq 3\eta d_{ij}pm$.
  \end{CLAIM}
  \begin{proof}
    Let $Q'_i$ be the set of vertices in $V_i$ that are covered by some edge in $M'_{ij} \cup M'_{ik}$ and note that $Q'_i \supset Q_i$ (similarly define $Q_j'$ and $Q_k'$). For a vertex $w \in V_j$, let $\mathcal{B}_w$ denote the event that $w \in Q'_j$. Since $(V_1, V_2, V_3)$ is $(\delta, \varepsilon, p)$-super-regular, $\deg(w, V_i) \leq (1+\varepsilon)d_{ij}pm$ and $\deg(w, V_k) \leq (1+\varepsilon)d_{jk}pm$. Hence,
    \[
    P(\mathcal{B}_w) \leq \deg(w,V_j)q_{ij} + \deg(w,V_k)q_{ik} \leq 2(1+\varepsilon)\eta.
    \]
    The events $\{\mathcal{B}_w \colon w \in V_j\}$ are mutually independent and $|N_j| \geq (\delta/2)pm$, so by Chernoff's inequality, $|N_j \cap Q'_j| \leq (5\eta/2)|N_j|$ with probability at least $1 - e^{-c\eta\delta pm}$ for some absolute positive constant $c$. It follows that
    \[
    \deg(v, Q_j) \leq \deg(v, Q'_j) = |N_j \cap Q'_j| \leq (5\eta/2)(1+\varepsilon)d_{ij}pm \leq 3\eta d_{ij}pm
    \]
    with probability $1 - o(m^{-1})$.
  \end{proof}

  \begin{CLAIM}
    \label{claim_M3}
    For each fixed $v \in V'_i$, with probability $1 - o(m^{-1})$, the pair $(N_j, N_k)$ contains at least $(\eta / 2)d_{ij}d_{ik}p^2m$ edges of $M_{jk}$.
  \end{CLAIM}
  \begin{proof}
    Without loss of generality, we may assume that $(i,j,k) = (1,2,3)$. Since $v$ is $\varepsilon$-typical, $(1+\varepsilon^2)(1+2\varepsilon/\delta) \leq 2$, and $5\varepsilon + 2\varepsilon/\delta \leq 1/7$, Proposition~\ref{prop_goodtriangles} implies that there are sets $N''_2 \subset N_2$ and $N''_3 \subset N_3$ such that $(N''_2, N''_3)$ contains at least $(6/7)d_{12}d_{13}d_{23}p^3m^2$ edges, no vertex in $N''_2$ has more than $2d_{13}d_{23}p^2m$ neighbours in $N''_3$, and vice versa, no vertex in $N''_3$ has more than $2d_{12}d_{23}p^2m$ neighbours in $N''_2$. Since at most $\eta d_{12}d_{13}d_{23}p^3m^2$
    edges among $(N''_2, N''_3)$ belong to $E_{23}$ by either (A1) or (A2), it follows that
    \[
    \Ex[| M'_{23} \cap (N''_2, N''_3) | ] \geq (5/7)d_{12}d_{13}d_{23}p^3m^2q_{23} = (5/7)d_{12}d_{13}\eta p^2m.
    \]
    Since $|M'_{23} \cap (N''_2, N''_3)|$ is a sum of independent indicator random variables, Chernoff's inequality implies that for some absolute constant $c$,
    \[
    P\left( | M'_{23} \cap (N''_2, N''_3) | \geq (4/7)d_{12}d_{13}\eta p^2m \right) \geq 1 - e^{-c\eta\delta^2p^2m} \geq 1 - 1/m^2,
    \]
    provided that $C$ is sufficiently large.

    In order to estimate $|M_{23} \cap (N''_2, N''_3)|$, note that $|M'_{23} \cap (N''_2, N''_3)| - |M_{23} \cap (N''_2, N''_3)|$ is at most the number of vertices in $N''_2 \cup N''_3$ that are incident to an edge in $M'_{23} \cap (N''_2, N''_3)$ and some other edge in $M'_{12} \cup M'_{13} \cup M'_{23}$. Let $\mathcal{C}_w$ denote the event that $w$ is such a ``bad'' vertex. If $w \in N''_2$, then
    \begin{align*}
      P(\mathcal{C}_w) & \leq \deg(w, N''_3)q_{23} \cdot (\deg(w,V_3)q_{23} + \deg(w,V_1)q_{12}) \\
      & \leq 2d_{13}d_{23}p^2mq_{23} \cdot ((1+\varepsilon)d_{23}pmq_{23} + (1+\varepsilon)d_{12}pmq_{12}) = (1+\varepsilon)4\eta^2d_{13}\eta^2p.
    \end{align*}
    Since $|N''_2| \leq |N_2| \leq (1+\varepsilon)d_{12}pm$, the expected number of such ``bad'' vertices in $N''_2$ is at most $(1+\varepsilon)^24\eta^2d_{12}d_{23}p^2m$. The events $\{\mathcal{C}_w \colon w \in N''_2\}$ are mutually independent, so by Chernoff's inequality, for some absolute constant $c$, with probability at least $1 - e^{-c\delta^2\eta^2p^2m}$, there are at most $5\eta^2d_{12}d_{13}p^2m$ ``bad'' vertices in $N''_2$ and similarly, there are at most $5\eta^2d_{12}d_{13}p^2m$ ``bad'' vertices in $N''_3$. Hence, with probability $1 - o(m^{-1})$,
    \[
    |M_{23} \cap (N''_2, N''_3)| \geq (4/7 - 10\eta)d_{12}d_{13}\eta p^2m \geq (\eta/2)d_{12}d_{13}p^2m,
    \]
    provided that $C$ is sufficiently large.
  \end{proof}

  \begin{CLAIM}
    \label{claim_M4}
    With probability $1 - o(1)$, the endpoints of every $\eta$-good edge in $(V_j, V_k)$ have at least $(1-4\eta)d_{ij}d_{ik}p^2m$ common neighbours in $V_i \setminus Q_i$.
  \end{CLAIM}
  \begin{proof}
    For an arbitrary vertex $v \in V_i$, let $\mathcal{D}_v$ denote the event that $v \in Q'_i$. Clearly,
    \[
    P(\mathcal{D}_v) \leq \deg(v,V_j)q_{ij} + \deg(v,V_k)q_{ik} \leq 2(1+\varepsilon)\eta.
    \]
    Fix some $\eta$-good edge in $(V_j, V_k)$ and let $A \subset V_i$ be the set of common neighbours of its endpoints. Then $\BBE[|A \cap Q'_i|] = \sum_{v \in A} P(\mathcal{D}_v) \le |A| \cdot 2(1+\varepsilon)\eta$.
    Moreover by definition, $|A| \geq (1-\eta)d_{ij}d_{ik}p^2m \geq
    (1/2)\delta^2p^2m$. Since the events $\mathcal{D}_v$ are mutually
    independent, Chernoff's inequality implies that
    \[
    P(|A \cap Q'_i| \geq 3\eta|A|) \leq e^{-c\delta^2\eta p^2m}
    \]
    for some absolute positive constant $c$. Hence, if $C$ is sufficiently large, then with probability at least $1 - 1/m^3$,
    \[
    |A \setminus Q_i| \geq |A \setminus Q'_i| \geq (1-3\eta)|A| \geq (1-3\eta)(1-\eta)d_{ij}d_{ik}p^2m \geq (1-4\eta)d_{ij}d_{ik}p^2m.
    \]
    Since there are at most $m^2$ good edges, the claim is proved.
  \end{proof}

  Finally, note that Claims~\ref{claim_M1}--\ref{claim_M4} imply that (M1)--(M4) are satisfied with probability $1 - o(1)$ (one needs to apply the union bound over all choices of
  vertices in order to deduce (M2) and (M3) from \ref{claim_M2} and \ref{claim_M3}).
\end{proof}


Below is the main theorem of this section. It says that we can partition our graph into balanced strong-super-regular triples, a collection of vertex-disjoint triangles, and a set of at most $O(p^{-2})$ exceptional vertices. In the next section, we will prove that each of those strong-super-regular triples contains a triangle-factor.

\begin{THM}
  \label{thm_decomposition}
  For an arbitrary positive $\gamma$, there exists a positive $\delta$ such that for all $\varepsilon$, there exist constants $C$, $D$, and $\xi$ that satisfy the following. If $p \geq C(\log n/n)^{1/2}$, then a.a.s.~every $G \subset G(n,p)$ with $\delta(G) \geq (2/3+\gamma)np$ contains a subgraph $G' \subset G$ whose vertex set can be partitioned into sets $B$, $S$, and $(W'_i)_{i=1}^{3k}$, where $k \leq D$, such that
  \begin{enumerate}[(i)]
    \setlength{\itemsep}{1pt} \setlength{\parskip}{0pt}
    \setlength{\parsep}{0pt}
  \item
    $|B| \le Dp^{-2}$,
  \item
    $G[S]$ contains a perfect triangle packing,
  \item
    $(W'_{3t-2}, W'_{3t-1}, W'_{3t})$ is a $(\delta, \varepsilon,p)$-strong-super-regular in $G'$ for all $t \in [k]$, and
  \item
    $|W'_{3t-2}| = |W'_{3t-1}| = |W'_{3t}| \ge \xi n$ for all $t \in [k]$.
  \end{enumerate}
\end{THM}

\begin{proof}
  Let $\delta = \min\{\delta_{\ref{thm_cleanstage1}}(\gamma)/2, 3/4\}$. Without loss of generality, we may assume that $\varepsilon \leq 2/\delta$. Furthermore, let $\varepsilon_3 = \varepsilon_{\ref{prop_typicalstability}}(\varepsilon, 2\delta)$, $\varepsilon_1 = \min\{\varepsilon\delta^2/1180, \varepsilon_3\delta^2/40\}$, $\varepsilon_2 = \min\{\varepsilon_{\ref{prop_manygoodvertices}}(\varepsilon_1, 2\delta), \varepsilon_{\ref{thm_cleanstage1}}(\gamma), \varepsilon_{\ref{lemma_randommatching}}(\delta), \varepsilon/2\}$, and $\xi = \xi_{\ref{thm_cleanstage1}}(\varepsilon_2)/2$.
Let $\eta = 12\varepsilon_1/\delta^2$.
Moreover, let
\[ D = (6/\xi) \max\{D_{\ref{thm_cleanstage1}(\varepsilon_2)}, D_{\ref{prop_nontypicalvertices} (ii)}(1/4, \varepsilon_1\xi), D_{\ref{prop_randomgraphproperties2}}(\varepsilon_1\varepsilon_2 \delta^3\xi^2/4, 1/4) \} \]
and $C = \max\{ C_{\ref{thm_cleanstage1}}(\varepsilon_2), C_{\ref{prop_smallexpansion}}(3\varepsilon_1\xi) \}$.

  By Theorem~\ref{thm_cleanstage1}, there exists a further subgraph $G'$ of $G$ whose vertex set can be partitioned into sets $B_0$, $S_0$, and $(W_i)_{i=1}^{3k}$ such that $|B_0| \leq (D/2)p^{-2}$, $G[S_0]$ contains a perfect triangle packing, and for all $t \in \{1, \ldots, k\}$, the triple $(W_{3t-2}, W_{3t-1}, W_{3t})$ is  $(2\delta, \varepsilon_2, p)$-super-regular in $G'$ and satisfies $|W_{3t-2}| = |W_{3t-1}| = |W_{3t}|$. Moreover, the endpoints of no edge in $(W_{3t-2}, W_{3t-1})$ have more than $4|W_{3t}|p^2$ common neighbours in $W_{3t}$ (and a similar statement holds for other choices of indices). We will show that each such triple contains a slightly smaller $(\delta, \varepsilon, p)$-strong-super-regular triple in such a way that all but at most $O(p^{-1})$ leftover vertices can be covered by vertex-disjoint triangles. Obviously, this will imply the assertion of the theorem.

  Without loss of generality, we will only consider the triple $(W_1, W_2, W_3)$. For the sake of brevity, let $m = |W_1| = |W_2| = |W_3|$ and note that $m \geq 2\xi n$. Without loss of generality, we can condition on the event that $G(n,p)$ satisfies the assertions of
  \begin{itemize}
  \item
    Proposition~\ref{prop_nontypicalvertices}~(ii) with $\rho = 1/4$ and $\xi = \varepsilon_1m/n$,
  \item
    Proposition~\ref{prop_randomgraphproperties2} with $\varepsilon = \varepsilon_1\varepsilon_2\delta^3m^2/(4n^2)$, and $\rho = 1/4$, and
  \item
    Proposition~\ref{prop_smallexpansion} with $\xi = 6\varepsilon_1m/(2n)$.
  \end{itemize}

  For each $i$, let $X_i \subset W_i$ be the collection of vertices that are not $\varepsilon_1$-good. By Proposition~\ref{prop_manygoodvertices}, $|X_i| \leq \varepsilon_1 m$ and we may assume that
$|X_i| = \varepsilon_1 m$. We perform the following cleaning-up procedure. While constantly updating the sets $X_1$, $X_2$, and $X_3$, repeat the following. If there exists an $i$ and a vertex $v \in W_i \setminus X_i$ such that either
  \begin{enumerate}
  \item[(A)]
    $|N(v) \cap X_j| \geq 4\varepsilon_1pm$ for some $j$ or
  \item[(B)]
    the neighbourhood of $v$ contains more than $\varepsilon_2d_{12}d_{13}d_{23}p^3m^2$ edges whose endpoints have more than $5\varepsilon_1p^2m$ common neighbours in $X_i$,
  \end{enumerate}
  then move $v$ to $X_i$.

  \begin{CLAIM}
    The cleaning-up procedure finishes with $\varepsilon_1 m \le |X_i| \leq 3\varepsilon_1m$ for all $i$.
  \end{CLAIM}
  \begin{proof}
    Suppose that at some point in time, $|X_i| > 3\varepsilon_1m$ for some $i$, and consider the earliest such moment. Without loss of generality, we may assume that $i = 1$. Clearly, $3\varepsilon_1m < |X_1| \leq 4\varepsilon_1m$ and $|X_j| \leq 3\varepsilon_1m$ if $j \neq 1$. Since at the beginning, every $X_1$ contained at most $\varepsilon_1m$ vertices, $W_1$ contains either $\varepsilon_1m$ vertices satisfying (A) or $\varepsilon_1m$ vertices satisfying (B). The former is impossible, since $|X_j| \leq 3\varepsilon_1m$ for $j \neq 1$ and we assumed that $G(n,p)$ satisfies the assertion of Proposition~\ref{prop_nontypicalvertices}~(ii) with $\xi = \varepsilon_1m/n$. Since the endpoints of each edge in $(W_2, W_3)$ have at most $4p^2m$ common neighbours in $W_1$, the latter would imply that $(W_2, W_3)$ contains
    \[ (\varepsilon_1 m) \cdot (\varepsilon_2 d_{12}d_{13}d_{23}p^3 m^2) / (4p^2m) \ge  (\varepsilon_1\varepsilon_2\delta^3/4)pm^2 \]
    edges whose endpoints have more than $5\varepsilon_1p^2m$ common neighbours in $X_1$. Since $|X_1| \leq 4\varepsilon_1m$, this is impossible by our assumption that $G(n,p)$ satisfies the assertion of Proposition~\ref{prop_randomgraphproperties2} with $\varepsilon = \varepsilon_1\varepsilon_2\delta^3m^2/(4n^2)$, and $\rho = 1/4$.
  \end{proof}

  It is not hard to check that $(W_1 \setminus X_1, W_2 \setminus X_2, W_3 \setminus X_3)$ is $(\delta/2, \varepsilon', p)$-strong-super-regular. Unfortunately, this conclusion does not help us at the moment as we first need to absorb $X_1 \cup X_2 \cup X_3$ into vertex-disjoint triangles and in the process of absorbing those vertices, we may use some vertices from the triple $(W_1 \setminus X_1, W_2 \setminus X_2, W_3 \setminus X_3)$.

  For every $i$, let $Y_i \subset X_i$ be the set of vertices in $X_i$ that have more than $4\varepsilon_1pm$ neighbours in $X_j$ for some $j$ with $j \neq i$. Since $|X_j| \leq 3\varepsilon_1m$ and we assumed that $G(n,p)$ satisfies the assertion of Proposition~\ref{prop_nontypicalvertices}~(ii) with $\xi = \varepsilon_1m/n$, then $|Y_i| \leq Dp^{-1}/(6k)$. By adding arbitrary vertices of $X_i$ to $Y_i$, we can guarantee that $|Y_1| = |Y_2| = |Y_3|$. For every $i$ and $j$, let $E_{ij} = (X_i, W_j) \cup (W_i, X_j)$. Since $(W_i, W_j)$ is $(\varepsilon_2, p)$-regular and $|X_i| \ge \varepsilon_1 |W_i| \ge \varepsilon_2 |W_i|$, we have
  \[
  |E_{ij}| \leq (d_{ij}+\varepsilon)p(|X_i||W_j|+|W_i||X_j|) \leq (1+\varepsilon/\delta)6\varepsilon_1d_{ij}pm^2 \leq \eta d_{ij}pm^2.
  \]
  Fix a vertex $v \in W_i \setminus Y_i$. We check that (A2) in Lemma~\ref{lemma_randommatching} is satisfied. Let $(N''_j, N''_k)$ be as in (A2) in Lemma~\ref{lemma_randommatching}. Since $|N''_j \cap X_j| \leq |N(v) \cap X_j| \leq 4\varepsilon_1pm$ and similarly, $|N''_k \cap X_k| \leq 4\varepsilon_1pm$, we have
  \[
  |E_{jk} \cap (N''_j, N''_k)| \leq 4\varepsilon_1pm (2d_{ij}d_{jk}p^2m + 2d_{ik}d_{jk}p^2m) \leq (16\varepsilon_1/\delta) d_{ij}d_{ik}d_{jk}p^3m^2 \leq \eta d_{ij}d_{ik}d_{jk}p^3m^2.
  \]
  Lemma~\ref{lemma_randommatching} implies that a.a.s.~for each $i$ and $j$, there exists an $M_{ij} \subset (W_i \setminus X_i, W_j \setminus X_j)$ such that (M1)--(M4) in Lemma~\ref{lemma_randommatching} are satisfied with $V_i = W_i$ and $V'_i = W_i \setminus Y_i$ for each $i$. Let $Q_i$ be defined as in Lemma~\ref{lemma_randommatching}.

  \begin{CLAIM}
    \label{claim_Xi-Yicover}
    The sets $X_1 \setminus Y_1$, $X_2 \setminus Y_2$, and $X_3 \setminus Y_3$ can be covered by vertex-disjoint triangles that use only vertices in $Q_1 \cup X_1$, $Q_2 \cup X_2$, and $Q_3 \cup X_3$.
  \end{CLAIM}
  \begin{proof}
    Since $M_{12} \cup M_{13} \cup M_{23}$ is a matching whose edges are not incident to any vertex in $X_1 \cup X_2 \cup X_3$, it suffices to show that for each $i$, $j$, and $k$, the vertices of $X_i \setminus Y_i$ can be paired with some $|X_i \setminus Y_i|$ edges of $M_{jk}$ to form vertex-disjoint triangles.

    Let $H$ be the bipartite graph on the vertex set $(X_i \setminus Y_i) \cup M_{jk}$, where a vertex $w \in X_i \setminus Y_i$ is adjacent to an edge $\{u, v\} \in M_{jk}$ if and only if $\{u, v, w\}$ is a triangle in $(W_1, W_2, W_3)$. Clearly, it suffices to prove that $H$ contains a matching that covers $X_i \setminus Y_i$. We check that Hall's condition holds in $H$. Fix an arbitrary non-empty set $S \subset X_i \setminus Y_i$. If $|N_H(S)| \leq |S|$, then there would be an $x$ with $1 \leq x = |S| \leq |X_i| \leq 3\varepsilon_1m$ such that $G(n,p)$ contains some $x$ independent edges and $x$ vertices, each of which is adjacent to both ends of at least $(\eta/2)\delta^2p^2m$ of those edges, see (M3) in Lemma~\ref{lemma_randommatching}. This would contradict our assumption that $G(n,p)$ satisfies the assertion of Proposition~\ref{prop_smallexpansion} with $\xi = \eta\delta^2m/(2n) = 6\varepsilon_1m/(2n)$. Hence, $|N_H(S)| > |S|$ for all non-empty $S \subset X_i \setminus Y_i$.
  \end{proof}

  Fix any such triangle packing and for each $i$, let $X'_i = X_i \cup T_i$, where $T_i \subset Q_i$ is the set of vertices in $W_i \setminus X_i$ that are covered by the triangle packing. Note that $|T_i| = |X_j \setminus Y_j| + |X_k \setminus Y_k|$. Let $W'_i = W_i \setminus X'_i$. Since for each $i$, $|X'_i| = |X_i| + |T_i| = |X_1 \setminus Y_1| + |X_2 \setminus Y_2| + |X_3 \setminus Y_3| + |Y_i|$ and $|Y_1| = |Y_2| = |Y_3|$, the sets $W'_1$, $W'_2$, and $W'_3$ have the same number of elements. Denote this number by $m'$ and note that $m' \geq m - 9\varepsilon_1m \geq m/2 \geq \xi n$.

  \begin{CLAIM}
    \label{claim_strong-super-regular}
    The triple $(W'_1, W'_2, W'_3)$ is $(\varepsilon, \delta, p)$-strong-super-regular.
  \end{CLAIM}
  \begin{proof}
    Since $(W_1, W_2, W_3)$ is $(\varepsilon/2, p)$-regular with density at least $2\delta p$ and $m' \geq m/2$, Proposition~\ref{prop_inheritregularity} implies that $(W'_1, W'_2, W'_3)$ is $(\varepsilon, p)$-regular with density at least $\delta p$. Fix an index $i$, recall that $|X'_i| \leq 9\varepsilon_1m \leq \varepsilon_3m$, and let $v$ be an arbitrary vertex in $W'_i$. Without loss of generality, we may assume that $i = 1$. Since $v \not\in X_1$, (A) implies that $\deg(v,X_j) \leq 4\varepsilon_1pm$ for every $j$. Moreover, (M2) in Lemma~\ref{lemma_randommatching} implies that $\deg(v,Q_j) \leq 3\eta d_{ij}pm$. Hence,
    \[
    \deg(v, X'_j) \leq \deg(v, X_j) + \deg(v, T_j) \leq \deg(v, X_j) + \deg(v, Q_j) \leq \varepsilon_3pm,
    \]
    and by Proposition~\ref{prop_typicalstability}, $v$ becomes $\varepsilon$-typical in $(W'_1, W'_2, W'_3)$. It remains to show that $v$ is also $\varepsilon$-good. Since $v \not\in X_i$, it was $\varepsilon_1$-good in $(W_1, W_2, W_3)$ and it satisfies (B). Hence, the endpoints of all but at most $(\varepsilon_1+\varepsilon_2)d_{12}d_{13}d_{23}p^3m^2$ edges in the neighbourhood of $v$ have at least $(1-\varepsilon_1-5\varepsilon_1/\delta^2)d_{12}d_{13}p^2m$ common neighbours in $W_1 \setminus X_1$.
    Moreover by (M4), they have at most $4\eta d_{12}d_{13}p^2m$ common neighbours in $Q_1$. Since $1 -\varepsilon_1-5\varepsilon_1/\delta^2- 4\eta \geq (1+\varepsilon_2/\delta)^2(1-\varepsilon)$, each such edge is $\varepsilon$-good in the new triple. It follows that $v$ is $\varepsilon$-good.
  \end{proof}

  Finally, let $B = B_0 \cup \bigcup_{i=1}^{3k}Y_i$ and let $S = S_0 \bigcup_{i=1}^{3k}(T_i \cup (X_i \setminus Y_i))$. Clearly, the sets $B$, $S$, and $(W'_i)_{i=1}^{3k}$ partition the vertex set of $G$,
  \[
  |B| \leq |B_0| + \sum_{i=1}^{3k} |Y_i| \leq Dp^{-2}/2 +  3k \cdot Dp^{-1}/(6k) \leq Dp^{-2}/2,
  \]
  and $G[S]$ contains a perfect triangle packing. Finally, by Claim~\ref{claim_strong-super-regular}, for each $t \in [k]$, the triple $(W'_{3t-2}, W'_{3t-1}, W'_{3t})$ is $(\varepsilon, \delta, p)$-strong-super-regular and satisfies $|W'_{3t-2}| = |W'_{3t-1}| = |W'_{3t}| \geq \xi n$.
\end{proof}

\section{Perfect triangle packing in strong-super-regular triples}
\label{section_packingssregular}

In the previous section, we managed to decompose the graph into balanced strong-super-regular triples, a triangle packing, and a small set of exceptional vertices. In this section, we will show how to find a triangle-factor in each of those triples. We start this section by showing how to construct sets of ``buffer'' vertices and edges that will allow us to complete an almost-spanning triangle packing into a triangle-factor.

\begin{LEMMA} \label{lemma_main}
  For all positive constants $\delta$, $\xi$, and $\eta$ with $\eta \leq 1/140$, there exist constants $C(\delta,\eta,\xi)$ and $\varepsilon(\delta,\eta)$ such that if $p \geq C(\log n/n)^{1/2}$, then $G(n,p)$ a.a.s.~satisfies the following. Let $(W_1,W_2,W_3)$ be a $(\delta, \varepsilon, p)$-strong-super-regular triple in a subgraph of $G(n,p)$ such that $|W_1| = |W_2| = |W_3| \geq \xi n$. Then there exist edge sets $M_{12}$, $M_{13}$, $M_{23}$ and vertex sets $X_1$, $X_2$, $X_3$ with the following properties:
  \begin{enumerate}
  \item[(P1)]
    $M_{12} \cup M_{13} \cup M_{23}$ is a matching.
  \item[(P2)]
    For all $j$ and $k$, $M_{jk} \subset (W_j,W_k)$ and $(\eta/2) |W_j| \le |M_{jk}| \le 2\eta |W_j| $.
  \item[(P3)]
    For all $i$, $|X_i| \leq (\eta / 4)|W_i|$ and $X_i \subset W_i \setminus Q_i$, where $Q_i$ is the set of vertices in $W_i$ that are covered by some edge in $M_{ij} \cup M_{ik}$.
  \item[(P4)]
    \label{item_factorZi}
    For all $i$, $j$, and $k$, if $Z_i \subset W_i$ has size $|M_{jk}|$ and contains $X_i$, then the subgraph of $(W_1,W_2,W_3)$ induced by $Z_i$ and $M_{jk}$ contains a triangle-factor.
  \end{enumerate}
\end{LEMMA}
\begin{proof}
  For the sake of brevity, let $m = |W_1| = |W_2| = |W_3|$. Without loss of generality, we may assume that $\delta \leq 1$. Let $\alpha = \delta^2/24$ and let $\beta$ be a positive constant satisfying $\beta \log(e/\beta) < \alpha\eta/30$. Moreover, let $\varepsilon = \min\{\delta/2, \eta/4, \varepsilon_{\ref{lemma_randommatching}}(\delta), \beta \cdot (\varepsilon_0)_{\ref{prop_randomgraphtypical}}(\alpha\delta/16,\delta/2), \alpha\beta\delta/16\}$ and let $\varepsilon' = 4\varepsilon/\delta$. Finally, let $C$ be sufficiently large so that $C(\log n/n)^{1/2} \geq C_{\ref{lemma_randommatching}}(\log m/m)^{1/2}$ and without loss of generality we may assume that $G(n,p)$ satisfies the assertion of Proposition~\ref{prop_smallexpansion} with $\xi = \eta\delta^2m/(2n)$ and $\xi = \eta\delta^2m/(12n)$, and Proposition~\ref{prop_randomgraphtypical} with $\varepsilon'_{\ref{prop_randomgraphtypical}}=\alpha\delta/16$ and $\delta_{\ref{prop_randomgraphtypical}} = \delta/2$.

  For all $i$ and $j$, let $q_{ij} = \eta/(d_{ij}mp)$ and select each $\varepsilon'$-good edge of $(W_i,W_j)$ independently with probability $q_{ij}$. Let $M'_{ij}$ be the set of all selected edges in $(W_i,W_j)$ and let $M_{ij} \subset M'_{ij}$ be the set of all those edges that are not incident to any other selected edge. By Proposition~\ref{prop_goodedges}, $(W_i, W_j)$ contains at most $\eta d_{ij}pm^2$ edges that are not $\varepsilon'$-good. Since each $v \in W_i$ is $\varepsilon$-good, its neighbourhood contains at most $\eta d_{12}d_{13}d_{23}p^3m^2$ edges that are not $\varepsilon'$-good. Therefore, Lemma~\ref{lemma_randommatching} applies with $E_{ij}$ being the set of non-$\varepsilon'$-good edges in $(W_i, W_j)$ and $V_i = V'_i = W_i$.

 \begin{CLAIM}
   \label{claim_largeexpansion}
    With probability $1 - o(1)$, every set $Y_i \subset W_i$ of size $\beta m$ satisfies the following. All but at most $\alpha\eta m$ edges of $M_{jk}$ belong to the neighbourhood of some vertex of $Y_i$.
  \end{CLAIM}
  \begin{proof}
    Fix a $Y_i \subset W_i$ of size $\beta m$. By Proposition \ref{prop_inheritregularity}, the triple $(Y_i, W_j, W_k)$ is $(\varepsilon/\beta, p)$-regular, and the densities of all three of its parts are at least $\delta/2$. Moreover, since we assumed that $G(n,p)$ satisfies the assertion of Proposition~\ref{prop_randomgraphtypical}, $(Y_i, W_j, W_k)$ is $\alpha\delta/16$-typical and by our assumption on $\varepsilon$, it is $(\alpha\delta/16,p)$-regular. By Proposition~\ref{prop_goodedges}, all but at most $(\alpha/2)d_{23}pm^2$ edges between $W_j$ and $W_k$ are $\alpha/2$-good, so in particular all but at most $(\alpha/2)d_{23}pm^2$ edges in $(W_j, W_k)$ belong to the neighbourhood of some vertex in $Y_i$. Hence the expected number of edges chosen among those ``bad'' edges is at most $(\alpha/2)d_{23}pm^2 q_{23} = (\alpha \eta /2) m$. Chernoff's inequality implies that the probability that more than $\alpha\eta m$ of those edges are chosen to $M'_{jk}$ is at most $e^{-\alpha\eta m/30}$. Since there are ${m \choose \beta m}$ $\beta m$-subsets of $W_i$,
    \[
    {m \choose \beta m} \leq \left(\frac{em}{\beta m}\right)^{\beta m} = e^{\beta\log(e/\beta)m},
    \]
    and $\alpha\eta/30 > \beta\log(e/\beta)$, the probability that all sets $Y_i$ have the claimed property is $1 - o(1)$.
  \end{proof}

  Lemma~\ref{lemma_randommatching} and Claim~\ref{claim_largeexpansion} imply that there exist $M_{12}$, $M_{13}$, and $M_{23}$ such that $M_{12} \cup M_{13} \cup M_{23}$ is a matching and for all $i$, $j$, and $k$ (properties 1, 2, and 3 follow from (M1), (M3), and (M4) of Lemma \ref{lemma_randommatching}, respectively, whereas property 4 follows from Claim \ref{claim_largeexpansion}):
 \begin{enumerate}
  \item
    $(\eta/2)m \leq |M_{jk}| \le 2\eta m$,
  \item \label{item_degreeWiMjk}
    the neighbourhood of every vertex in $V_i$ contains at least $(\eta/2)\delta^2p^2m$ edges of $M_{jk}$,
  \item
    the endpoints of each edge of $M_{jk}$ have at least $(1/2)\delta^2p^2m$ common neighbours in $W_i \setminus Q_i$,
  \item \label{item_expandYi}
    for every set $Y_i \subset W_i$ of size $\beta m$, all but at most $\alpha\eta m$ edges of $M_{jk}$ belong to the neighbourhood of some vertex of $Y_i$.
  \end{enumerate}

  Fix any such $M_{12}$, $M_{13}$, and $M_{23}$. Next, for each $i$, let $X_i$ be a random binomial subset of $W_i \setminus Q_i$, where each element is included with probability $\eta/5$. A simple application of Chernoff's inequality combined with Property 3 above shows that if $C$ is sufficiently large, then with probability $1 - o(1)$, for all $i$, $j$, and $k$:
  \begin{enumerate}
    \setcounter{enumi}{4}
  \item \label{item_sizeXi}
    $|X_i| \leq (\eta/4)m$,
  \item \label{item_degreeMjkXi}
    the endpoints of each edge of $M_{jk}$ have at least $(\eta/12)\delta^2p^2m$ common neighbours in $X_i$.
  \end{enumerate}
  Let $X_1$, $X_2$, and $X_3$ be arbitrary sets satisfying \ref{item_sizeXi} and \ref{item_degreeMjkXi} and note that properties (P1)--(P3) are satisfied. It remains to show that (P4) is also satisfied.

  Fix a $Z_i \subset W_i$ of size $|M_{jk}|$ such that $X_i \subset Z_i$. Let $H$ be the bipartite graph on the vertex set $Z_i \cup M_{jk}$, where a vertex $w \in Z_i$ is adjacent to an edge $\{u,v\} \in M_{jk}$ if and only if $\{u,v,w\}$ is a triangle in $(W_1,W_2,W_3)$. Clearly, it suffices to prove that $H$ contains a perfect matching. We check that $H$ satisfies the assumptions of Proposition \ref{prop_Hall} with $A = Z_i$, $B = M_{jk}$, and $L = \alpha\eta m$.

  Fix an $S \subset Z_i$. If $0 < |S| \leq (\eta/4)\delta^2m$, then $|N_H(S)| > |S|$ or otherwise there would be an $x \in [1, (\eta/4)\delta^2m]$ such that $G(n,p)$ contains some $x$ independent edges and $x$ vertices, each of which is adjacent to both ends of at least $(\eta/2)\delta^2p^2m$ of those edges, see~\ref{item_degreeWiMjk}. This would contradict our assumption that $G(n,p)$ satisfies the assertion of Proposition~\ref{prop_smallexpansion} with $\xi = \eta\delta^2m/(2n)$. On the other hand, if $|S| \geq (\eta/4)\delta^2m \geq \beta m$, then by \ref{item_expandYi}, $|M_{jk} \setminus N_H(S)| \leq \alpha\eta m$. Hence, $|N_H(S)| \geq |S|$ as long as $|Z_i \setminus S| \geq \alpha\eta m$

  Finally, fix a $T \subset M_{jk}$ with $0 < |T| \leq \alpha\eta m = (\eta/24)\delta^2m$. If $|N_H(T)| < |T|$, then there would be an $x \in [1,(\eta/24)\delta^2m]$ such that $G(n,p)$ contains $x$ vertices and $x$ independent edges whose endpoints have at least $(\eta/12)\delta^2p^2m$ common neighbours among those $x$ vertices (recall that $X_i \subset Z_i$), see~\ref{item_degreeMjkXi}. This would contradict our assumption that $G(n,p)$ satisfies the assertion of Proposition~\ref{prop_smallexpansion} with $\xi = \eta\delta^2m/(12n)$.
\end{proof}

With Lemma~\ref{lemma_main} at hand, without much effort we can prove
Theorem \ref{thm_triangleblowup} which says that a balanced strong-super-regular
triple has a triangle-factor.

\begin{proof}[Proof of Theorem \ref{thm_triangleblowup}.]
Let $\eta=1/140$, $\varepsilon_1 =
\varepsilon_{\ref{prop_findtriangle}}(3\eta/2, \delta/2)$, and
$\varepsilon=(\eta/4)\min\{\varepsilon_{\ref{lemma_main}}(\delta,
\eta), \varepsilon_1\}$. Let $C =  C_{\ref{lemma_main}}(\delta,
\eta, \xi)$.
Let $(W_1, W_2, W_3)$ be a $(\delta, \varepsilon,
p)$-strong-super-regular triple and $m = |W_1| = |W_2| = |W_3|$. By
Lemma~\ref{lemma_main}, there exists a matching $M_{12}, M_{13},
M_{23}$ and sets $Q_1, Q_2, Q_3$ and $X_1, X_2, X_3$ satisfying
(P1), (P2), (P3), and (P4). Let $W_i' = W_i \setminus (Q_i \cup X_i)$
and note that
\[ |W_1'| = |W_1| - |Q_1| - |X_1| = |W_1| - |M_{12}| - |M_{13}| - |X_1|. \]
Let $x = m - |M_{12}| - |M_{13}| - |M_{23}|$ and note that $(1 - 6\eta)m < x < (1 - 3\eta/2)m$.
By applying Proposition~\ref{prop_findtriangle}, we can find $x$ vertex-disjoint
triangles inside the triple $(W_1', W_2', W_3')$.

Note that the remaining $|W_1'| - x = |M_{23}| - |X_1|$ vertices in $W_1'$ together with the set $X_1$, can be matched with the set $M_{23}$ to construct $|M_{23}|$ vertex-disjoint triangles, by property (P4). Similarly, the remaining vertices in $W_2' \cup X_2$ and $W_3' \cup X_3$ can be
matched with $M_{13}$ and $M_{12}$, respectively. Therefore, we have found a perfect triangle packing
of $(W_1, W_2, W_3)$.
\end{proof}

Finally, we briefly summarize Sections~\ref{section_cleanup1}, \ref{section_cleanup2}, and \ref{section_packingssregular} in the proof of our main result, Theorem~\ref{thm_main}.

\begin{proof}[Proof of Theorem~\ref{thm_main}]
  Let $\delta = \delta_{\ref{thm_decomposition}}(\gamma)$, $\varepsilon = \varepsilon_{\ref{thm_triangleblowup}}(\delta)$, $\xi = \xi_{\ref{thm_decomposition}}(\delta,\varepsilon)$ and $C=\max\{C_{\ref{thm_decomposition}}(\varepsilon),C_{\ref{thm_triangleblowup}}(\varepsilon, \xi)\},  D=D_{\ref{thm_decomposition}}(\varepsilon)$.
By Theorem \ref{thm_decomposition}, there exist set $B$, $S$, and $(W_i)_{i=1}^{3k}$ which
satisfies (i) - (iv) of Theorem \ref{thm_decomposition}. Furthermore for each $t \in [k]$, by Theorem \ref{thm_triangleblowup},
each $(\delta,\varepsilon,p)$-strong-super-regular triple $(W_{3t-2}, W_{3t-1}, W_{3t})$ contains
a perfect triangle packing. Therefore all the vertices except $B$ can be covered by
vertex-disjoint triangles. Since $|B| \le Dp^{-2}$, this completes the proof.
\end{proof}

\section{Concluding Remarks}
\label{section_concluding-remarks}

An immediate question we would like to ask is whether the assumption on $p$ in Theorem~\ref{thm_main} can be relaxed. Even though our argument breaks down (for a few reasons) if $p \ll (\log n/n)^{1/2}$, we believe that the conclusion of Theorem~\ref{thm_main} still holds under the (weaker) assumption that $p \gg n^{-1/2}$. If this was true, it would completely resolve the problem of determining the local resilience of $G(n,p)$ with respect to the property of containing an almost spanning triangle packing.

We also believe that a similar argument can be used to obtain an extension of the theorem of Hajnal and Szemer{\'e}di~\cite{HaSz} for larger cliques to the setting of sparse random graphs. Clearly, the edge probability $p$ would have to be sufficiently large so that a corresponding form of Lemma~\ref{lemma_typicalregulartriple} holds. However, in our opinion, the importance of such a result does not justify the technical complications one would have to face in order to prove it.

The more intriguing and interesting question comes from the attempt to embed general spanning or almost spanning
graphs (by general we mean graphs that are not disjoint unions of a fixed graph) into sparse regular pairs.
This gives rise to the following question.

\begin{QUES}
  Can we develop an embedding lemma for general graphs into
  regular pairs in random graphs for some $p=n^{-o(1)}$? How
  should the definition of strong-super-regularity be extended?
\end{QUES}

It is quite likely that such an embedding lemma will provide another proof of the
theorem of B\"{o}ttcher, Kohayakawa, and Taraz~\cite{BoKoTa} on embedding almost spanning subgraphs.
However, one can hope for a better result where the graph we want to embed is smaller than
the host graph by a sublinear number of vertices.
To achieve this, one will most likely need to develop a tool similar to that of Theorem~\ref{thm_decomposition}.

Another question can be asked regarding embedding of spanning subgraphs.
Proposition~\ref{prop_removetriangles} shows that as many as $\Omega(p^{-2})$ vertices
have to be left out from the largest triangle packing. More generally, if every vertex
of some graph $H$ is contained in a triangle, then we cannot hope to embed $H$ into
a sparse host graph of the same order. However, this is no longer the case when $H$ is
bipartite. Thus we recall the following question posed by B{\"o}ttcher, Kohayakawa,
and Taraz~\cite{BoKoTa09}.

\begin{QUES}
  Is it possible to have a perfect embedding for bipartite graphs?
\end{QUES}

In fact, it might be true that what actually matters is not that the graph is bipartite,
but the fact that there are enough vertices which are not contained in a copy of a triangle.
See~\cite{HuLeSu}, where such a result is proved for dense graphs.

\medskip

\noindent {\bf Acknowledgements.}
We are indebted to the anonymous referee for their extremely careful reading of the paper and many helpful comments and suggestions.
The bulk of this work was done when the second and the third authors were visiting the first author in the Department of Mathematics at the University of California, San Diego.

\bibliography{TrianglePacking}{}

\providecommand{\bysame}{\leavevmode\hbox to3em{\hrulefill}\thinspace}
\providecommand{\MR}{\relax\ifhmode\unskip\space\fi MR }
\providecommand{\MRhref}[2]{%
  \href{http://www.ams.org/mathscinet-getitem?mr=#1}{#2}
}
\providecommand{\href}[2]{#2}
\begin{thebibliography}{10}

\bibitem{AlSp}
N.~Alon and J.~Spencer, \emph{The probabilistic method}, third ed., John Wiley
  \& Sons Inc., Hoboken, NJ, 2008.

\bibitem{AlYu-almost}
N.~Alon and R.~Yuster, \emph{Almost {$H$}-factors in dense graphs}, Graphs and
  Combinatorics \textbf{8} (1992), 95--102.

\bibitem{AlYu-Gnp}
\bysame, \emph{Threshold functions for {$H$}-factors}, Combinatorics,
  Probability and Computing \textbf{2} (1993), 137--144.

\bibitem{AlYu-factor}
\bysame, \emph{{$H$}-factors in dense graphs}, Journal of Combinatorial Theory.
  Series B \textbf{66} (1996), 269--282.

\bibitem{BaCsSa}
J.~Balogh, B.~Csaba, and W.~Samotij, \emph{Local resilience of almost spanning
  trees in random graphs}, Random Structures \& Algorithms \textbf{38} (2011),
  121--139.

\bibitem{BSKrSu}
S.~Ben-Shimon, M.~Krivelevich, and B.~Sudakov, \emph{Local resilience and
  {H}amiltonicity {M}aker-{B}reaker games in random regular graphs},
  Combinatorics, Probability and Computing \textbf{20} (2011), 173--211.

\bibitem{BoKoTa}
J.~B\"{o}ttcher, Y.~Kohayakawa, and A.~Taraz, \emph{Almost spanning subgraphs
  of random graphs after adversarial edge removal}, arXiv:1003.0890v1
  [math.CO].

\bibitem{BoKoTa09}
J.~B{\"o}ttcher, Y.~Kohayakawa, and A.~Taraz, \emph{Problem session},
  Combinatorics and Probability Workshop, Oberwolfach (2009).

\bibitem{CoHa}
K.~Corr{\'a}di and A.~Hajnal, \emph{On the maximal number of independent
  circuits in a graph}, Acta Mathematica Academiae Scientiarum Hungaricae
  \textbf{14} (1963), 423--439.

\bibitem{DeKoMaSt}
D.~Dellamonica, Y.~Kohayakawa, M.~Marciniszyn, and A.~Steger, \emph{On the
  resilience of long cycles in random graphs}, Electronic Journal of
  Combinatorics \textbf{15} (2008), Research Paper 32, 26.

\bibitem{Dirac}
G.~A. Dirac, \emph{Some theorems on abstract graphs}, Proceedings of the London
  Mathematical Society \textbf{2} (1952), 69--81.

\bibitem{ErRe}
P.~Erd{\H{o}}s and A.~R{\'e}nyi, \emph{On the evolution of random graphs},
  Magyar Tud. Akad. Mat. Kutat\'o Int. K\"ozl. \textbf{5} (1960), 17--61.

\bibitem{FrKr}
A.~Frieze and M.~Krivelevich, \emph{On two {H}amilton cycle problems in random
  graphs}, Israel Journal of Mathematics \textbf{166} (2008), 221--234.

\bibitem{GeSt}
S.~Gerke and A.~Steger, \emph{The sparse regularity lemma and its
  applications}, Surveys in combinatorics 2005, London Math. Soc. Lecture Note
  Ser., vol. 327, Cambridge Univ. Press, Cambridge, 2005, pp.~227--258.

\bibitem{HaSz}
A.~Hajnal and E.~Szemer{\'e}di, \emph{Proof of a conjecture of {P}. {E}rd{\H
  o}s}, Combinatorial theory and its applications, {II} ({P}roc. {C}olloq.,
  {B}alatonf\"ured, 1969), North-Holland, Amsterdam, 1970, pp.~601--623.

\bibitem{Hall}
P.~Hall, \emph{On representation of subsets}, Journal of the London
  Mathematical Society \textbf{10} (1935), 26--30.

\bibitem{HuLeSu}
H.~Huang, C.~Lee, and B.~Sudakov, \emph{Bandwidth theorem for sparse graphs},
  to appear in Journal of Combinatorial Theory B,
  DOI:10.1016/j.jctb.2011.03.002.

\bibitem{JoKaVu}
A.~Johansson, J.~Kahn, and V.~Vu, \emph{Factors in random graphs}, Random
  Structures \& Algorithms \textbf{33} (2008), 1--28.

\bibitem{Kim}
J.~H. Kim, \emph{Perfect matchings in random uniform hypergraphs}, Random
  Structures \& Algorithms \textbf{23} (2003), 111--132.

\bibitem{Kohayakawa}
Y.~Kohayakawa, \emph{Szemer\'edi's regularity lemma for sparse graphs},
  Foundations of computational mathematics ({R}io de {J}aneiro, 1997),
  Springer, Berlin, 1997, pp.~216--230.

\bibitem{KoLuRo}
Y.~Kohayakawa, T.~{\L}uczak, and V.~R{\"o}dl, \emph{On {$K^4$}-free subgraphs
  of random graphs}, Combinatorica \textbf{17} (1997), 173--213.

\bibitem{KoRo03}
Y.~Kohayakawa and V.~R{\"o}dl, \emph{Regular pairs in sparse random graphs.
  {I}}, Random Structures \& Algorithms \textbf{22} (2003), 359--434.

\bibitem{KoRo}
\bysame, \emph{Szemer\'edi's regularity lemma and quasi-randomness}, Recent
  advances in algorithms and combinatorics, CMS Books Math./Ouvrages Math. SMC,
  vol.~11, Springer, New York, 2003, pp.~289--351.

\bibitem{KoSaSz97}
J.~Koml{\'o}s, G.~S{\'a}rk{\"o}zy, and E.~Szemer{\'e}di, \emph{Blow-up lemma},
  Combinatorica \textbf{17} (1997), 109--123.

\bibitem{KoSaSz}
\bysame, \emph{Proof of the {A}lon-{Y}uster conjecture}, Discrete Mathematics
  \textbf{235} (2001), 255--269, Combinatorics (Prague, 1998).

\bibitem{Krivelevich}
M.~Krivelevich, \emph{Triangle factors in random graphs}, Combinatorics,
  Probability and Computing \textbf{6} (1997), 337--347.

\bibitem{KrLeSu}
M.~Krivelevich, C.~Lee, and B.~Sudakov, \emph{Resilient pancyclicity of random
  and pseudorandom graphs}, SIAM Journal on Discrete Mathematics \textbf{24}
  (2010), 1--16.

\bibitem{KrSuSz}
M.~Krivelevich, B.~Sudakov, and T.~Szab{\'o}, \emph{Triangle factors in sparse
  pseudo-random graphs}, Combinatorica \textbf{24} (2004), 403--426.

\bibitem{KuOs09}
D.~K{\"u}hn and D.~Osthus, \emph{Embedding large subgraphs into dense graphs},
  Surveys in combinatorics 2009, London Math. Soc. Lecture Note Ser., vol. 365,
  Cambridge Univ. Press, Cambridge, 2009, pp.~137--167.

\bibitem{KuOs}
\bysame, \emph{The minimum degree threshold for perfect graph packings},
  Combinatorica \textbf{29} (2009), 65--107.

\bibitem{LeSu}
C.~Lee and B.~Sudakov, \emph{Dirac's theorem for random graphs},
  arXiv:1108.2502v2 [math.CO].

\bibitem{Posa}
L.~P{\'o}sa, \emph{Hamiltonian circuits in random graphs}, Discrete Mathematics
  \textbf{14} (1976), 359--364.

\bibitem{Rucinski}
A.~Ruci{\'n}ski, \emph{Matching and covering the vertices of a random graph by
  copies of a given graph}, Discrete Mathematics \textbf{105} (1992), 185--197.

\bibitem{SuVu}
B.~Sudakov and V.~H. Vu, \emph{Local resilience of graphs}, Random Structures
  \& Algorithms \textbf{33} (2008), 409--433.

\bibitem{Szemeredi}
E.~Szemer{\'e}di, \emph{Regular partitions of graphs}, Probl\`emes
  combinatoires et th\'eorie des graphes ({C}olloq. {I}nternat. {CNRS}, {U}niv.
  {O}rsay, {O}rsay, 1976), Colloq. Internat. CNRS, vol. 260, CNRS, Paris, 1978,
  pp.~399--401.

\end{thebibliography}
\bibliographystyle{amsplain}

\end{document}